\documentclass{amsart}
\usepackage{amsmath,mathrsfs}
  \usepackage{paralist,enumerate}
  \usepackage{graphics} 
  \usepackage{epsfig} 
\newcommand{\no}{\nonumber}

\newcommand{\e}{\varepsilon}

\allowdisplaybreaks

\newtheorem{theorem}{Theorem}[section]

\newtheorem{lemma}[theorem]{Lemma}
\newtheorem{proposition}[theorem]{Proposition}

\theoremstyle{definition}

\title[A fully nonlinear equation for the flame front]
{A fully nonlinear equation for the flame front in a quasi-steady
combustion model}

\author[C.-M. Brauner, J. Hulshof, L. Lorenzi and G. I. Sivashinsky]{}


 \email{brauner@math.u-bordeaux1.fr}
 \email{jhulshof@cs.vu.nl}
 \email{luca.lorenzi@unipr.it}
\email{grishas@post.tau.ac.il}

\begin{document}

\maketitle

\centerline{\scshape Claude-Michel Brauner}
\medskip
{\footnotesize
 \centerline{Institut de Math\'ematiques de Bordeaux, Universit\'e de Bordeaux}
 \centerline{33405 Talence cedex, France}
 \centerline{and}
 \centerline{Department of Mathematics, Xiamen University}
 \centerline{361005 Xiamen, China}
} 

\medskip

\centerline{\scshape Josephus Hulshof}
\medskip
{\footnotesize
 \centerline{Faculty of Sciences, Mathematics and Computer Sciences Division}
\centerline{VU University Amsterdam}
\centerline{1081 HV Amsterdam, The
Netherlands}
} %

\medskip

\centerline{\scshape Luca Lorenzi}
\medskip
{\footnotesize
 \centerline{Dipartimento di Matematica,
Universit\`a degli Studi di Parma}
   \centerline{Viale G. Usberti 53/A, 43124 Parma, Italy}
} %

\medskip

\centerline{\scshape Gregory I. Sivashinsky}
\medskip
{\footnotesize
 \centerline{School of Mathematical Sciences, Tel Aviv University}
 \centerline{69978 Tel Aviv, Israel}
} %
\medskip
\medskip
\centerline{\dedicatory{\emph{\small This paper is dedicated to Roger Temam on the occasion of his 70th birthday}}}
%

\date{\today}

\begin{abstract}
We revisit the Near Equidiffusional Flames (NEF) model introduced by Matkowsky
and Sivashinsky in 1979 and consider a simplified, quasi-steady version of it.
This simplification allows, near the planar front, an explicit derivation of the front equation. The latter
is a pseudodifferential fully nonlinear parabolic equation of the fourth-order. First, we study the
(orbital) stability of the null solution. Second, introducing a parameter $\e$,
we rescale both the dependent and independent
variables and prove rigourously the convergence to the solution of
the Kuramoto-Sivashinsky equation as $\e \to 0$.
\end{abstract}

\section{Introduction}
Flames constitute a complex physical system involving fluid
dynamics, multistep chemical kinetics, as well as molecular and
radiative transfer. The laminar flames of low-Lewis-number
premixtures are known to display diffusive-thermal instability
responsible for the formation of a non-steady cellular structure
(see \cite{S83}). However, the cellular instability is quite robust
against these aero-thermo-chemical complexities and may be
successfully captured by a model involving only two equations: the heat
equation for the system's temperature and the diffusion equation for
the deficient reactant's concentration. In suitably chosen units,
the so-called thermal-diffusional model reads, see e.g.,
\cite{BuckLud}:
\begin{align}
\Theta_{t}  &  =\Theta_{xx}+\Theta_{yy}+\Omega(Y,\Theta),\label{I1}\\
Y_{t}  &  =Le^{-1}(Y_{xx}+Y_{yy})-\Omega(Y,\Theta),\label{I2}\\
\Omega &  =\frac{1}{2}Le^{-1}\beta^{2}Y\exp[\beta(\Theta-1)/(\sigma
+(1-\sigma)\Theta)]. \label{I3}%
\end{align}
Here, $\Theta=(T-T_{u})/(T_{ad}-T_{u})$ is the scaled temperature,
where $T_{u}$ and $T_{ad}$ correspond to the temperature of the
unburned gas, and the adiabatic temperature of combustion products,
respectively; $Y=C/C_{u}$ is the scaled concentration of the
deficient reactant with $C_{u}$ being its value in the unburned gas;
$x,y,t$ are the scaled spatiotemporal coordinates referred to
$D_{th}/U$, and $D_{th}/U^{2}$, respectively, where $D_{th}$ is the
thermal diffusivity of the mixture and $U$ is the velocity of the
undisturbed planar flame; $Le$ is the Lewis number (the ratio of
thermal and molecular diffusivities); $\sigma=T_{u}/T_{ad}$;
$\beta=T_{a}(1-\sigma)/T_{ad}$ is the Zeldovich number, assumed to
be large, where $T_{a}$ is the activation temperature; $\Omega$ is
the scaled reaction rate, where the normalizing factor
$\frac{1}{2}Le^{-1}\beta^{2}$ ensures that at $\beta\gg1$ the planar
flame propagates at the velocity close to unity.

Due to the distributed nature of the reaction rate $\Omega$,
Equations \eqref{I1} and \eqref{I2} are still difficult for a
theoretical exploration. One therefore turns to the conventional
high activation energy limit ($\beta\gg 1$) which converts the
reaction rate term into a localized source distributed over a
certain interface $x=\xi(t,y)$, the flame front. Intensity of the
source varies along the front as $\exp\left
(\frac{1}{2}(\Theta_f-1)\right )$ (see \cite{S77}). Here, $\Theta_f$
is the scaled temperature at the curved front, which may differ from
unity $(T=T_{ad})$ by a quantity of the order of $\beta^{-1}$. Due
to the strong temperature dependence of the reaction rate ($\beta\gg
1$), even slight changes of $\Theta_f$ may markedly affect its
intensity, and thereby also local flame speed. The study of flame
propagation is thus reduced to a free-interface problem. To ensure
that the emerging free-interface model does not involve large
parameters one should combine the limit of large activation energy
($\beta\gg 1$) with the requirement that the product
$\alpha=\frac{1}{2}\beta(1-Le)$ remains finite, i.e., the ratio of
thermal and molecular diffusivities ($Le$) should be closed to unity.
This is the Near Equidiffusive Flames model, in short NEF, introduced
in \cite{mat-siva}. As a result, instead of the reaction diffusion problem
for $\Theta$ and $Y$, one ends up with a free-interface problem for the new scaled
temperature $\theta=\lim_{\beta\to +\infty}\Theta$ and the reduced
enthalpy $S=\lim_{\beta\to +\infty}\beta^{-1}(\Theta+Y-1)$. More
precisely, the system for the temperature $\theta$, the enthalpy
$S$ and the moving flame front, defined by $x = \xi (t,y)$, reads
\begin{align}
\label{eqn:FBP1} &\frac{\partial\theta}{\partial t}
= \Delta \theta, \;\;x < \xi (t,y),\\[1mm]
\label{eqn:FBP2}
&\theta = 1,\;\; x \geq \xi (t,y),\\
\label{eqn:FBP3} &\frac{\partial S}{\partial t} = \Delta S  -
\alpha\Delta\theta, \;\; x \neq \xi (t,y).
\end{align}
For some mathematical results about this problem, see
\cite{BL,lorenzi-1,lorenzi-2,lorenzi-3,lorenzi-4,DucrotMarion}. Here, we consider
only the case when $\alpha$ is positive, i.e., $Le<1$. It will be
convenient to assume periodicity in $y$ with period $\ell$, and
restrict attention to $y \in [-\ell/2, \ell/2]$. At the front,
$\theta$ and $S$ are continuous, the following jump conditions occur
for the normal derivatives:
\begin{align}
\label{eqn:FBP4}
&\bigg[\frac{\partial\theta}{\partial n}\bigg] = - \exp (S),\\[1mm]
\label{eqn:FBP5} &\bigg[\frac{\partial S}{\partial n}\bigg] = \alpha
\bigg[\frac{\partial\theta}{\partial n}\bigg].
\end{align}
System (\ref{eqn:FBP1})-(\ref{eqn:FBP5}) admits a planar
travelling wave (TW)
solution, with velocity $- 1$,
\begin{eqnarray*}
\overline{\theta}(x) = \left\{
\begin{array}{ll}
\exp x, &x\le 0,\\
1, & x\ge 0
\end{array}
\right. \qquad \overline{S}(x)= \left\{
\begin{array}{ll}
\alpha x \exp x, & x \leq 0,\\
0, & x \geq 0.
\end{array}
\right.
\end{eqnarray*}
As usual one fixes the free boundary. We set $\xi(t,y)=-t +
\varphi(t,y)$, $x'=x-\xi(t,y)$. In this new framework:
\begin{align}
\label{temp-1}
&\theta_t +(1-\varphi_t) \theta_{x'} = \Delta_\varphi \theta,\qquad\;\, x'<0,\\[1mm]
\label{temp-2}
&\theta(x') =1,\qquad\;\, x'>0,\\[1mm]
\label{enth} &S_t +(1-\varphi_t) S_{x'} = \Delta_\varphi S - \alpha
\Delta_\varphi \theta,\qquad\;\,x'\neq 0,
\end{align}
where
\begin{eqnarray*}
\Delta_\varphi = (1 + (\varphi_y)^2) D_{x'x'} + D_{yy}
-\varphi_{yy}D_{x'} -2\varphi_y D_{x'y}.
\end{eqnarray*}
The front is now fixed at $x'=0$.  The first condition
(\ref{eqn:FBP4}) reads:
\begin{eqnarray*}
\sqrt{1+(\varphi_y)^2}\,\,\bigg[\frac{\partial\theta}{\partial
x'}\bigg] = - \exp (S),
\end{eqnarray*}
the second one (\ref{eqn:FBP5}) becomes
\begin{eqnarray*}
\bigg[\frac{\partial S}{\partial x'}\bigg] = \alpha
\bigg[\frac{\partial\theta}{\partial x'}\bigg].
\end{eqnarray*}

A very challenging problem is the derivation of a single equation
for the interface or moving front $\varphi$, which may capture most
of the dynamics and, as a consequence, yields a reduction of the
effective dimensionality of the system. In this spirit, one of the
authors in \cite{siva} derived asymptotically from the System
\eqref{eqn:FBP1}-\eqref{eqn:FBP5} the Kuramoto-Sivashinsky (K--S)
equation in rescaled dependent and independent variables
\begin{equation}
\label{eqn:KS} \Phi_{\tau} + 4\Phi_{\eta\eta\eta\eta} +
\Phi_{\eta\eta} + \frac{1}{2}({\Phi_{\eta})}^2 = 0.
\end{equation}
Since then, this equation has received considerable attention from the mathematical
community. We refer to the book \cite{temam} and its extensive bibliography.

This paper is devoted to a quasi-steady version of
the NEF model. As a matter of fact, it has been observed in similar
problems (see \cite{BFHS}) that not far from the instability threshold the
time derivatives in the temperature and enthalpy equations have a
relatively small effect on the solution. The dynamics appears to be
essentially driven by the front. Based on this observation one can
define a \textit{quasi-steady} NEF model replacing
(\ref{temp-1})-(\ref{enth}) by
\begin{align*}
&(1-\varphi_t) \theta = \Delta_\varphi \theta,\qquad\;\, x'<0,\\[1.5mm]
&\theta =1,\qquad x'>0,\\[1.5mm]
&(1-\varphi_t) S_{x'} = \Delta_\varphi S - \alpha \Delta_\varphi
\theta,\qquad\, x'\neq 0.
\end{align*}
Next we consider the perturbations of temperature $u$ and enthalpy
$v$:
\begin{eqnarray*}
\theta = \overline{\theta}+u, \quad S = \overline{S}+ v.
\end{eqnarray*}
Writing for simplicity $x$ instead of $x'$, the problem for the
triplet $(u,v,\varphi)$ reads:
\begin{align*}
&(1-\varphi_t)u_x -\Delta_\varphi u - \varphi_t \overline{\theta}_x
= (\Delta_\varphi
- \Delta)\overline{\theta},\qquad x<0,\\[1mm]
&u=0,\qquad x>0,\\[1mm]
&(1-\varphi_t)v_x -\Delta_\varphi (v-\alpha u) - \varphi_t
\overline{S}_x = (\Delta_\varphi - \Delta)(\overline{S} -\alpha
\overline{\theta}),\quad x \neq 0,
\end{align*}
where
\begin{align*}
&(\Delta_\varphi - \Delta)(\overline{\theta}) = ((\varphi_y)^2 -
\varphi_{yy})\overline{\theta}_x,\\[2mm]
&(\Delta_\varphi - \Delta)(\overline{S} -\alpha \overline{\theta}) =
\alpha ((\varphi_y)^2 \overline{S}_x - \varphi_{yy}\overline{S}).
\end{align*}
As in \cite{BHL08,BHL09}, we introduce further simplifications: we
keep only linear and second-order terms for the perturbation of the
front $\varphi$, and first-order terms for the perturbations of
temperature $u$ and enthalpy $v$. This leads to the
equations:
\begin{align*}
&u_x -\Delta u - \varphi_t \overline{\theta}_x = (\Delta_\varphi -
\Delta)\overline{\theta}, \qquad x<0,\\[2mm]
&v_x -\Delta(v-\alpha u) - \varphi_t \overline{S}_x =
(\Delta_\varphi - \Delta)(\overline{S} -\alpha \overline{\theta}),
\quad x\neq0.
\end{align*}
At $x=0$ there are several conditions. First
\begin{eqnarray*}
[u]=[v]=0,
\end{eqnarray*}
however, since $u(x)=0$ for $x>0$, this is equivalent to
\begin{eqnarray*}
u(0^-)=[v]=0.
\end{eqnarray*}
Second,
\begin{eqnarray*}
\sqrt{1 + (\varphi_y)^2}[\overline{\theta}_x +u_x] =
-\exp(\overline{S}+v),
\end{eqnarray*}
hence up to the second-order:
\begin{eqnarray*}
-1 + [u_x]=-(1 +(\varphi_y)^2)^{-\frac{1}{2}}\,e^v \sim \left(1-
\frac{1}{2}(\varphi_y)^2\right)\left(1+v(0) + \frac{1}{2}
(v(0))^2\right)
\end{eqnarray*}
and keeping only the first-order for $v$ yields:
\begin{align*}
&-u_x(0) +v(0)= \frac{1}{2}(\varphi_y)^2,\\
&[v_x]=-\alpha u_x(0).
\end{align*}
Therefore, the final system reads:
\begin{equation}
\left\{
\begin{array}{ll}
u_x -\Delta u - \varphi_t \overline{\theta}_x =((\varphi_y)^2 -
\varphi_{yy})\overline{\theta}_x, & x<0,\\[1mm]
v_x -\Delta (v-\alpha u) - \varphi_t \overline{S}_x = (\varphi_y)^2
\overline{S}_x -
\varphi_{yy}\overline{S}, & x \neq 0,\\[1mm]
u(0)=[v]=0,\\[1mm]
v(0)-u_x(0) = \frac{1}{2}(\varphi_y)^2,\\[1mm]
{\rm [}v_x{\rm ]}=-\alpha u_x(0).
\end{array}
\right. \label{final-u}
\end{equation}
We remark that the equation for $u$ associated with the boundary
condition $u(0)=0$ entirely determines $u$ when $\varphi$ is given.
Therefore, it can be viewed as a kind of {\it pseudo-differential
Stefan condition}. We will take advantage of this remark in Section
\ref{sect-3}.

The goal of this paper is to show that this simplified NEF model
still contains the dynamics of the system. It is simple enough to be integrated
explicitly via a discrete Fourier transform in the
variable $y$ and therefore it allows a separation of the dependent variables.
We get to a self-consistent pseudo-differential
equation for the front $\varphi$ which reads:
\begin{align}
&(X^2_k+\alpha X_k-\alpha)\widehat\varphi_t(t,k)
=(-4\lambda_k^2+(\alpha-1)\lambda_k)\widehat\varphi(t,k)\notag\\
&+\frac{1}{4}(X^3_k-3X_k^2-4\alpha
X_k+4\alpha)\widehat{(\varphi_y)^2}(t,k), \quad k=0,1,\ldots,
\label{psd-front}
\end{align}
where the $-\lambda_k$'s are the non-positive eigenvalues of the
operator $D_{yy}$ with periodic boundary conditions at $y=\pm
\ell/2$ (that we denote below by $A$) and
\begin{eqnarray*}
X_k=\sqrt{1+4\lambda_k},\qquad\;\,k=0,1,\ldots,
\end{eqnarray*}
is the symbol of operator $\sqrt{1-4D_{yy}}$.

Equation \eqref{psd-front} can be written in the more abstract form:
\begin{equation}
\label{psd-abstract-0} \frac{\partial}{\partial t}\varphi={\mathscr
L}(\varphi)+{\mathscr G}((\varphi_y)^2),
\end{equation}
where ${\mathscr L}$ is a pseudodifferential operator whose leading
part is $D_{yy}$ and $G$ is a nonlinear operator whose leading term
is $\frac{1}{4}\sqrt{1-4D_{yy}}$. This makes
$\eqref{psd-abstract-0}$ a strongly nonlinear equation, more
precisely it is a {\it fully nonlinear parabolic equation}: in the
$L^2$-setting the nonlinear part is exactly of the same order as the
linear operator. This is one of the main issues of this
paper.
Note that the
realization of the operator $\sqrt{1-4D_{yy}}$ in the space of
continuous and $\ell$-periodic functions (say $C_{\sharp}$) is
defined only in a {\it proper} subspace of $C^1_{\sharp}$ (the space
of all the $\ell$-periodic $C^1$-functions). Hence, in the
$C_{\sharp}$-setting, the nonlinear term ${\mathscr G}((\varphi_y)^2)$
represents the leading part of the right-hand side of
\eqref{psd-abstract-0}. This would make the study of
\eqref{psd-abstract-0} more difficult than in the $L^2$-setting,
where we confine our analysis.

In the case when $\varphi$ is smoother, we can rewrite Equation
\eqref{psd-abstract-0} as a fourth-order equation as follows:
\begin{eqnarray}\label{psd-abstract}
\frac{\partial}{\partial t}{\mathscr B}\varphi = {\mathscr S}(\varphi)
+ {\mathscr F}((\varphi_y)^2),
\end{eqnarray}
where ${\mathscr S}$ is nothing but the usual fourth-order
differential operator
\begin{eqnarray*}
{\mathscr S}(\varphi) = -\varphi_{yyyy} -(\alpha-1)\varphi_{yy}.
\end{eqnarray*}
Operators ${\mathscr B}$ and ${\mathscr F}$ are pseudo-differential
ones with symbols, respectively,
\begin{eqnarray*}
b_k = X^2_k+\alpha X_k-\alpha, \quad f_k=
\frac{1}{4}(X^3_k-3X_k^2-4\alpha X_k+4\alpha).
\end{eqnarray*}
Therefore,
\begin{align*}
{\mathscr B}&=I-4D_{yy}+\alpha\left (\sqrt{I-4D_{yy}}-I\right ),\\
{\mathscr
F}&=\frac{1}{4}(I-4D_{yy})^{\frac{3}{2}}-\frac{3}{4}(I-4D_{yy})-\alpha\left
(\sqrt{1-4D_{yy}}-I\right ).
\end{align*}
The main feature of Equation \eqref{psd-abstract} is that the
nonlinear part is rather unusual. Actually, it has a fourth-order
leading term, as ${\mathscr S}$ has. Therefore  \eqref{psd-abstract}
is also a fully nonlinear problem, in contrast to \eqref{eqn:KS}
which is semilinear.

The paper is organized as follows: In Section \ref{sect-3} we derive
the front equation via an explicit computation of \eqref{final-u} in
the strip ${\mathbb R} \times [-\ell/2,\ell/2]$. Then, in Section
\ref{sect-4}, we prove the following result
\begin{theorem}
\label{thm:1} Let
\begin{equation}
\alpha_c=1+\frac{16\pi^2}{\ell^2}. \label{eq-c}
\end{equation}
Then, the following properties are met.
\begin{enumerate}[\rm (a)]
\item
If $\alpha<\alpha_c$, then, the null solution to Equation
\eqref{psd-abstract} is $($orbitally$)$ stable, with asymptotic
phase, with respect to sufficiently smooth and small perturbations.
\item
If $\alpha>\alpha_c$, then the null solution to Equation
\eqref{psd-abstract} is unstable.
\end{enumerate}
\end{theorem}

An important question, that we address in Section \ref{sect-5}, is
the link between \eqref{psd-abstract} and K--S. Following
\cite{siva}, we introduce a small parameter $\varepsilon>0$, setting
\begin{equation*}
\alpha = 1+ \varepsilon,
\end{equation*}
and define the rescaled dependent and independent variables
accordingly:
\begin{eqnarray*}
t= \tau/\varepsilon^2,\quad y = \eta/\sqrt{\varepsilon}, \quad
\varphi = \varepsilon \psi.
\end{eqnarray*}
We see that $\psi$ solves the equation
\begin{align*}
&\frac{\partial}{\partial\tau}\left\{I-4\e D_{\eta\eta}+(1+\e)\left (\sqrt{I-4\e D_{\eta\eta}}-1\right )\right\}\psi\\[1mm]
=&-4D_{\eta\eta\eta\eta}\psi-D_{\eta\eta}\psi\\[1mm]
&+\frac{1}{4}\left\{(I-4\e D_{\eta\eta})^{\frac{3}{2}}-3(I-4\e
D_{\eta\eta})-4(1+\e)\left (\sqrt{1-4\e D_{\eta\eta}}-I\right )
\right\}(D_{\eta}\psi)^2.
\end{align*}

Then, we anticipate, in the limit $\varepsilon \to 0$, that $\psi
\sim \Phi$, where $\Phi$ solves \eqref{eqn:KS}. More precisely, we
take for $\ell$:
\begin{equation*}
\ell_{\varepsilon}= \ell_0/\sqrt{\varepsilon},
\end{equation*}
which blows up as $\varepsilon \to 0$; hence
$\alpha_c=1+\frac{16\pi^2}{\ell_0^2}\,\varepsilon$. Thus, $\ell_0$
becomes the new bifurcation parameter. We shall assume that
$\ell_0>4\pi$ in order to have $\alpha_c \in (1,1+\varepsilon)$,
i.e., $\alpha>\alpha_c$, otherwise the trivial solution is stable
and the dynamics is trivial.

The second main result of the paper is the following.
\begin{theorem}
\label{thm:2} Let $\Phi_0 \in H^m$ be a periodic function of period
$\ell_0$. Further, let $\Phi$ be the periodic solution of
\eqref{eqn:KS} $($with period $\ell_0)$ on a fixed time interval
$[0,T]$, satisfying the initial condition $\Phi(0,\cdot)=\Phi_0$.
Then, if $m$ is large enough, there exists $\e_0=\e_0(T)\in (0,1)$
such that, for $0<\e\leq\e_{0}$, Problem \eqref{psd-abstract} admits
a unique classical solution $\varphi$ on $[0,T/\varepsilon^2]$,
which is periodic with period $\ell_{0}/\sqrt{\varepsilon}$ with
respect to $y$, and satisfies
\begin{eqnarray*}
\varphi(0,y)=\e\Phi_0(y\sqrt{\varepsilon}),\qquad\;\,|y|\le
\frac{\ell_0}{2\sqrt{\e}}.
\end{eqnarray*}
Moreover, there exists a positive constant $C$, independent of
$\e\in (0,\varepsilon_0]$, such that
\begin{eqnarray*}
|\varphi(t,y)-\varepsilon
\Phi(t\varepsilon^{2},y\sqrt{\varepsilon})|\leq
C\,\varepsilon^{2},\qquad\;\, 0\leq
t\leq\frac{T}{\varepsilon^{2}},\;\,|y|\leq\frac{\ell_{0}}{2\sqrt{\varepsilon}},
\end{eqnarray*}
for any $\varepsilon\in (0,\e_0]$.
\end{theorem}
In other words, starting from the same configuration, the solution
of \eqref{psd-abstract} remains on a fixed time interval close to
the solution of K--S up to some renormalization, uniformly in
$\varepsilon$ sufficiently small. Note that the initial condition
for $\varphi$ is of special type, compatible with $\Phi_{0}$\ and
\eqref{eqn:KS} at $\tau=0$. Initial conditions of this type have
been already considered in \cite{BFHLS,BFHS,BHL09}.

Although energy methods are known to be usually inefficient in fully
nonlinear problems, here we may take advantage of the special
structure of $\mathscr F$. It allows us to establish sharp a priori
estimates on the remainder (more precisely on its derivative) when $\e$
is small enough. A key point is an extension of a lemma that we
already successfully used in \cite{BFHLS,BHL09}.

Finally, in the Appendix, for the reader's convenience we provide a
quite detailed proof of the existence, uniqueness and regularity of
the solution to K--S which vanishes at $\tau=0$.

In a forthcoming paper we will incorporate the time derivatives
of the temperature and enthalpy in the model;
the front equation will be more involved
and of higher order in time, as in \cite{BFHLS}.
Another issue we intend to address is the derivation of the front equation
as a solvability condition in the spirit of \cite{BHL08,BHL09}.
\section{Some mathematical setting}\label{math-setting}
\setcounter{equation}{0}
In this section we introduce some notation, the functional spaces and operators we will use below.
We will mainly use the discrete Fourier transform with respect to the variable $y$. For this purpose,
given a function $f:(-\ell/2,\ell/2)\to\mathbb C$, we denote
by $\widehat f(k)$ its $k$-th Fourier coefficient, that is, we write
\begin{eqnarray*}
f(y)=\sum_{k=0}^{+\infty}\widehat f(k)w_k(y),\qquad\;\,y\in (-\ell/2,\ell/2),
\end{eqnarray*}
where $\{w_{k}\}$ is a complete set of (complex valued)
eigenfunctions of the operator
\begin{eqnarray*}
A=D_{yy}:D(A)={H}^{2}(-\ell/2,\ell/2)\,\to\, {L}^{2}(-\ell/2,\ell/2),
\end{eqnarray*}
with $\ell$-periodic boundary conditions,
corresponding to the non-positive eigenvalues
\begin{eqnarray*}
0,-\frac{4\pi^2}{\ell^2},-\frac{4\pi^2}{\ell^2},-\frac{16\pi^2}{\ell^2},-\frac{16\pi^2}{\ell^2},-\frac{36\pi^2}{\ell^2},\dots
\end{eqnarray*}
We shall find it convenient to label this sequence as
\begin{eqnarray*}
0=-\lambda_0(\ell)>-\lambda_1(\ell)=-\lambda_2(\ell)>-\lambda_3(\ell)=-\lambda_4(\ell)>\dots
\end{eqnarray*}
When there is no damage of confusion, we simply write $\lambda_k$
instead of $\lambda_k(\ell)$.

When $f$ depends also on $t$ and/or $x$, by $\widehat f(\cdot,k)$
we denote the $k$-th Fourier coefficient of $f$ with respect to $y$. For instance,
for fixed $t$ and $x$,
$\widehat f(t,x,k)$ will denote the $k$-th Fourier coefficient of the function $f(t,x,\cdot)$.

For integer or arbitrary real $s$, we denote by $H^{s}_{\sharp}$ the usual Sobolev
spaces of order $s$ consisting of $\ell$-periodic (generalized) functions, which we
will conveniently represent as
\begin{eqnarray*}
H_{\sharp}^{s}=\left\{w=\sum_{k=0}^{+\infty}a_{k} w_{k}:\,\sum_{k=0}^{+\infty}\lambda_{k}^{s}
a_{k}^{2}<+\infty\right\},
\end{eqnarray*}
with norm
\begin{eqnarray*}
\|w\|^{2}_{s}=\sum_{k=0}^{+\infty}\lambda_{k}^{s} a_{k}^{2}.
\end{eqnarray*}
For $k=0$, we simply write $L^2$ instead of $H^0_{\sharp}$ and
$|\cdot|_2$ instead of $\|\cdot\|_0$.

We recall that for any $\beta>0$ and $\gamma\in (0,1)$ the operator
$(I-\beta A)^{\gamma}$ has $H^{2\gamma}_{\sharp}$ as a domain and it
is defined by its symbol $((1+\beta\lambda_k)^{\gamma})$ (see e.g.,
\cite[Thm. 4.33]{lunardi-interp}).

Next, for any $n=0,1,\ldots$ and any $\beta\in [0,1)$, we set
\begin{eqnarray*}
C_{\sharp}^{n+\beta}=\{f\in C^{n+\beta}([-\ell/2,\ell/2]): f^{(j)}(-\ell/2)=f^{(j)}(\ell/2),~j\le n\}.
\end{eqnarray*}
$C^{n+\beta}_{\sharp}$ is endowed with the Euclidean norm of
$C^{n+\beta}([-\ell/2,\ell/2])$. Finally, we denote by
$\|\cdot\|_{\infty}$ the sup-norm.

\section{The derivation of a self-consistent equation for the front}
\label{sect-3} \setcounter{equation}{0} The aim of this section is
the derivation of a self-consistent equation (in the Fourier
variables) for the front $\varphi$. For this purpose, we rewrite
Problem \eqref{final-u}, making $\overline\theta$ and $\overline S$
explicit. We get
\begin{equation}
\left\{
\begin{array}{ll}
u_x -\Delta u =(\varphi_t+(\varphi_y)^2 -\varphi_{yy})e^x, & x<0,\\[1mm]
v_x -\Delta (v-\alpha u) = \alpha(\varphi_t+(\varphi_y)^2)(x+1)e^x -
\alpha\varphi_{yy}xe^x, & x< 0,\\[1mm]
v_x -\Delta v = 0, & x > 0,\\[1mm]
u(0)=[v]=0,\\[1mm]
v(0)-u_x(0) = \frac{1}{2}(\varphi_y)^2,\\[1mm]
{\rm [}v_x{\rm ]}=-\alpha u_x(0).
\end{array}
\right.
\label{final-u-1}
\end{equation}
In what follows, we assume that $(u,v,\varphi)$ is a sufficiently
smooth solution to Problem \eqref{final-u-1} such that
the function $x\mapsto e^{-x/2}u(t,x,y)$ is bounded in $(-\infty,0]$ and the function
$x\mapsto e^{-x/2}v(t,x,y)$ is bounded in $\mathbb R$.
As it has been stressed in the Introduction, we use the first equation
in \eqref{final-u-1} and the boundary condition
$u(\cdot,0,\cdot)=0$ as a pseudo-differential Stefan condition.
We solve the problem for $u$ via discrete Fourier transform.
This leads us to the infinitely many equations
\begin{equation}
\widehat u_x(t,x,k)-\widehat u_{xx}(t,x,k)+\lambda_k\widehat u(t,x,k)=\left
(\widehat\varphi_t(t,k)+\widehat{(\varphi_y)^2}(t,k)
+\lambda_k\widehat\varphi(t,k)\right )e^x,
\label{eq-uk}
\end{equation}
for $k=0,1,2,\ldots$, where we recall that $-\lambda_k=-\lambda_k(\ell)$ is the $k$-th eigenvalue of
the realization of the operator $D_{yy}$ in $L^2$. For notational
convenience we set $\nu_k=\frac{1}{2}+\frac{1}{2}\sqrt{1+4\lambda_k}$ for any $k=0,1,\ldots$
A straightforward computation reveals that the solution to \eqref{eq-uk} which
vanishes at $x=0$ and tends to $0$ as $x\to -\infty$ not slower than $e^{-x/2}$ is given by
\begin{align*}
\widehat u(t,x,0)=-\left (\widehat\varphi_t(t,0)+\widehat{(\varphi_y)^2}(t,0)\right )xe^x,\qquad\;\,x\le 0.
\end{align*}

Let us now consider the problem for $v$, where we disregard (for the moment) the condition
$v(\cdot,0,\cdot)-u_x(\cdot,0,\cdot)=\frac{1}{2}(\varphi_y)^2$.
Taking the Fourier transform (with respect to the variable $y$), we get the
Cauchy problems
\begin{eqnarray*}
\left\{
\begin{array}{ll}
\widehat v_x(t,x,0)-\widehat v_{xx}(t,x,0)=\alpha\left (\widehat\varphi_t(t,0)+
\widehat{(\varphi_y)^2}(t,0)\right )(2x+3)e^x,\quad & x<0,\\[1mm]
\widehat v_x(t,x,0)-\widehat v_{xx}(t,x,0)=0, & x>0,\\[1mm]
[\widehat v(t,\cdot,0)]=0,\\[1mm]
[\widehat v_x(t,\cdot,0)]=-\alpha\widehat u_x(t,0,0)=\alpha \left (
\widehat\varphi_t(t,0)+\widehat{(\varphi_y)^2}(t,0)\right ),
\end{array}
\right.
\end{eqnarray*}
for $k=0$, and
\begin{eqnarray*}
\left\{
\begin{array}{ll}
\quad\displaystyle\widehat v_x(t,x,k)-\widehat v_{xx}(t,x,k)+\lambda_k v(t,x,k)\\[3mm]
=\displaystyle\alpha \left (x+2-\frac{1}{\lambda_k}\right )\left (\widehat\varphi_t(t,k)+
\widehat{(\varphi_y)^2}(t,k)\right )e^x+\alpha\lambda_k\left (x+1-\frac{1}{\lambda_k}\right )\widehat\varphi(t,k)e^x\\[4mm]
\quad\displaystyle+\frac{\alpha\nu_k}{\lambda_k}\left (\widehat\varphi_t(t,k)+
\widehat{(\varphi_y)^2}(t,k)+\lambda_k\widehat\varphi(t,k)\right )e^{\nu_kx}, &\hskip -33pt x<0,\\[5mm]
\widehat v_x(t,x,k)-\widehat v_{xx}(t,x,k)+\lambda_k\widehat v(t,x,k)=0, & \hskip -33pt x>0,\\[1mm]
[\widehat v(t,\cdot,k)]=0,\\[1mm]
[\widehat v_x(t,\cdot,k)]=-\alpha\widehat u_x(t,0,k)=\alpha\nu_k^{-1}
\left (\widehat\varphi_t(t,k)+\widehat{(\varphi_y)^2}(t,k)
+\lambda_k\widehat\varphi(t,k)\right ),
\end{array}
\right.
\end{eqnarray*}
for $k\ge 1$.

It is easy to show that
\begin{align*}
&\widehat v(t,x,0)=-\alpha\left (\widehat\varphi_t(t,0)+\widehat{(\varphi_y)^2}(t,0)\right )x(x+1)e^x,\qquad\;\,x<0,\\[1mm]
&\widehat v(t,x,0)=0,\qquad\;\,x>0.
\end{align*}
and
\begin{align*}
\widehat v(t,x,k)=&c_{1,k}e^{\nu_kx}+\frac{\alpha}{\lambda_k}\left
(\widehat\varphi_t(t,k)+\widehat{(\varphi_y)^2}(t,k)\right )(x+2)e^x
+\alpha\widehat\varphi(t,k)(x+1)e^x\\
&+\frac{\alpha}{\lambda_k}\frac{\nu_k}{1-2\nu_k}\left
(\widehat\varphi_t(t,k)+\widehat{(\varphi_y)^2}(t,k)
+\lambda_k\widehat\varphi(t,k)\right )xe^{\nu_kx},\qquad\;\,x<0\\[3mm]
\widehat v(t,x,k)=&c_{2,k}e^{(1-\nu_k)x},\qquad\;\,x\ge 0,
\end{align*}
where
\begin{align*}
c_{1,k}=&\frac{\alpha}{1-2\nu_k}\left (1+\nu_k+\frac{\nu_k}{1-2\nu_k}+\frac{\lambda_k}{\nu_k}\right )
\widehat\varphi(t,k)\\
&+\frac{\alpha}{1-2\nu_k}\left (\frac{1}{\lambda_k}+\frac{2\nu_k}{\lambda_k}+\frac{1}{\nu_k}
+\frac{1}{\lambda_k}\frac{\nu_k}{1-2\nu_k}\right )\left (\widehat\varphi_t(t,k)+\widehat{(\varphi_y)^2}(t,k)\right ),\\[3mm]
c_{2,k}=&\frac{\alpha}{1-2\nu_k}\left (2+\frac{\nu_k}{1-2\nu_k}+\frac{\lambda_k}{\nu_k}-\nu_k\right )
\widehat\varphi(t,k)\\
&-\frac{\alpha}{1-2\nu_k}\left (\frac{2\nu_k}{\lambda_k}-\frac{3}{\lambda_k}-\frac{1}{\lambda_k}\frac{\nu_k}{1-2\nu_k}
-\frac{1}{\nu_k}\right )\left (\widehat\varphi_t(t,k)+\widehat{(\varphi_y)^2}(t,k)\right )
\end{align*}

Now, we are in a position to determine the equation for the front. Indeed, rewriting the boundary condition
\begin{eqnarray*}
v(0)-u_x(0)=\frac{1}{2}(\varphi_y)^2,
\end{eqnarray*}
in Fourier variables, and using the above results, we get to the
following equations for the front (in the Fourier coordinates):
\begin{align*}
&\widehat\varphi_t(t,0)+\frac{1}{2}\widehat{(\varphi_y)^2}(t,0)=0,\\[3mm]
&\left\{\frac{\alpha}{1-2\nu_k}\left (2+\frac{\nu_k}{1-2\nu_k}+\frac{\lambda_k}{\nu_k}-\nu_k\right )+\frac{\lambda_k}{\nu_k}\right\}
\widehat\varphi(t,k)\\[1mm]
&+\left\{-\frac{\alpha}{1-2\nu_k}\left (\frac{2\nu_k}{\lambda_k}-\frac{3}{\lambda_k}-\frac{1}{\lambda_k}\frac{\nu_k}{1-2\nu_k}
-\frac{1}{\nu_k}\right )+\frac{1}{\nu_k}\right\}\widehat\varphi_t(t,k)\\[1mm]
&+\left\{-\frac{\alpha}{1-2\nu_k}\left
(\frac{2\nu_k}{\lambda_k}-\frac{3}{\lambda_k}-\frac{1}{\lambda_k}\frac{\nu_k}{1-2\nu_k}
-\frac{1}{\nu_k}\right
)+\frac{1}{\nu_k}-\frac{1}{2}\right\}\widehat{(\varphi_y)^2}(t,k)=0.
\end{align*}
Let us set $X_k=\sqrt{1+4\lambda_k}$. Then, the equation for $\varphi$ reads (in terms of $X_k$) as follows:
\begin{align*}
&\frac{(X_k-1)(X_k^2-\alpha)}{2X_k^2}\widehat\varphi(t,k)+
\frac{2(X_k^2+\alpha X_k-\alpha)}{X^2_k(X_k+1)}\widehat\varphi_t(t,k)\\
-&\frac{X^3_k-3X_k^2-4\alpha X_k+4\alpha}{2X_k^2(X_k+1)}\widehat{(\varphi_y)^2}(t,k)=0,
\end{align*}
for any $k=0,1,2,\ldots$, or, equivalently,
\begin{equation}
4\widehat\varphi_t(t,k)=
\frac{(1-X_k^2)(X_k^2-\alpha)}{X^2_k+\alpha X_k-\alpha}\widehat\varphi(t,k) +\frac{X^3_k-3X_k^2-4\alpha
X_k+4\alpha}{X^2_k+\alpha X_k-\alpha}\widehat{(\varphi_y)^2}(t,k),
\label{eq-symb-2nd}
\end{equation}
or, even,
\begin{align*}
(X^2_k+\alpha X_k-\alpha)\widehat\varphi_t(t,k)=&\frac{1}{4}(1-X_k^2)(X_k^2-\alpha)\widehat\varphi(t,k)\notag\\
&+\frac{1}{4}(X^3_k-3X_k^2-4\alpha X_k+4\alpha)\widehat{(\varphi_y)^2}(t,k)\notag\\
=&(-4\lambda_k^2+(\alpha-1)\lambda_k)\widehat\varphi(t,k)\notag\\
&+\frac{1}{4}(X^3_k-3X_k^2-4\alpha
X_k+4\alpha)\widehat{(\varphi_y)^2}(t,k),
\end{align*}
for any $k=0,1,\ldots$ Hence, we can
say that $\varphi$ solves the equations
\begin{equation}
\frac{d}{dt}{\mathscr B}(\varphi)={\mathscr S}(\varphi)+{\mathscr
F}((\varphi_y)^2)
\label{eq-4-order}
\end{equation}
and
\begin{equation}
\varphi_t={\mathscr B}^{-1}{\mathscr S}(\varphi)+{\mathscr B}^{-1}{\mathscr F}((\varphi_y)^2):={\mathscr L}(\varphi)
+{\mathscr G}((\varphi_y)^2),
\label{second-order-varphi}
\end{equation}
where the operators ${\mathscr B}$, ${\mathscr S}$ and ${\mathscr F}$ are defined through their
symbols
\begin{align}
b_k&=X^2_k+\alpha X_k-\alpha,
\label{bk}
\\[1mm]
s_k&=-4\lambda_k^2+(\alpha-1)\lambda_k,\label{sk}
\\[1mm]
f_k&=\frac{1}{4}(X^3_k-3X_k^2-4\alpha X_k+4\alpha), \label{f0k}
\end{align}
for any $k=0,1,\ldots$

\section{Stability of the front}
\label{sect-4} \setcounter{equation}{0} In this section we are
interested in studying the stability and instability properties of the null solution
to the Equations \eqref{eq-4-order} and
\eqref{second-order-varphi}. In this respect we need to study the
symbols appearing in \eqref{eq-symb-2nd}.

\subsection{Study of the symbols}\label{symbol}

In this subsection, we study the main properties of the operators
${\mathscr B}$, ${\mathscr G}$, ${\mathscr S}$, and ${\mathscr L}$, ${\mathscr F}$,
whose symbols are respectively defined by \eqref{bk}-\eqref{f0k} and by
\begin{align*}
l_k&=\frac{(1-X_k^2)(X_k^2-\alpha)}{4(X^2_k+\alpha X_k-\alpha)},\\[1mm]
g_k&=\frac{X^3_k-3X_k^2-4\alpha X_k+4\alpha}{4(X^2_k+\alpha
X_k-\alpha)},
\end{align*}
for any $k=0,1,\ldots$ Even if all these operators depend on $\alpha$, we prefer not to stress
explicitly the dependence on $\alpha$ to avoid cumbersome notations.

\begin{proposition}
\label{prop-symb-1}
The following properties are met.
\begin{enumerate}[\rm (i)]
\item
The operator ${\mathscr L}$ admits a realization $L$ in $L^2$ which
is a sectorial operator. Moreover, its spectrum consists of the
sequence $(l_k)$. In particular, $0$ is a simple
eigenvalue of $L$. The spectral projection associated with this
eigenvalue is the operator $\Pi$ defined by
\begin{eqnarray*}
\Pi(\psi)=\frac{1}{\ell}\int_{-\frac{\ell}{2}}^{\frac{\ell}{2}}\psi(y)dy,\qquad\;\,\psi\in L^2.
\end{eqnarray*}
Finally, $\sigma(L)\setminus\{0\}\subset (-\infty,0)$
if and only if $\alpha<\alpha_c$ $($see \eqref{eq-c}$)$.
\item
The operator ${\mathscr B}$ admits a bounded realization $B$ mapping
$H^2_{\sharp}$ into $L^2$. Moreover, $B$ it is invertible.
\item
The operator ${\mathscr F}$ admits a bounded realization $F$ mapping
$H^3_{\sharp}$ into $L^2$.
\item
The operator ${\mathscr G}$ admits a bounded realization $G$ mapping
$H^1_{\sharp}$ into $L^2$.
\item
The realization of the operator ${\mathscr S}$ in $L^2$ is the
operator
\begin{align*}
S=-4D_{yyyy}+(\alpha-1)D_{yy},
\end{align*}
with $H^4_{\sharp}$ as domain.
\end{enumerate}
\end{proposition}

\begin{proof}
(i). To begin with, we observe that
\begin{eqnarray*}
l_k=-\frac{\lambda_k(4\lambda_k+1-\alpha)}{\alpha\sqrt{4\lambda_k+1}+4\lambda_k+1-\alpha}.
\end{eqnarray*}
Hence, we can split
\begin{align*}
l_k&=-\lambda_k+\frac{\alpha\lambda_k\sqrt{1+4\lambda_k}}{\alpha\sqrt{1+4\lambda_k}+4\lambda_k+1-\alpha}
:=-\lambda_k+l_{1,k},
\end{align*}
for any $k=0,1,\ldots$ Note that $l_{1,k}\sim \frac{\alpha}{2}\sqrt{\lambda_k}$ as $k\to +\infty$.
Hence, from the above splitting of the symbol $(l_k)$ it follows
at once that the operator ${\mathscr L}$ admits a realization $L$ in
$L^2$ with domain $D(L)=H^2_{\sharp}$ which can be split as
$L=A+L_1$, where $L_1$ is a bounded operator from $H^1_{\sharp}$
into $L^2$, and $A$ is the realization of $D_{yy}$ in $L^2$ with
domain $H^2_{\sharp}$. Since $H^1_{\sharp}$ is an intermediate space
of class $J_{1/2}$ between $L^2$ and $D(A)$, \cite[Prop.
2.4.1(i)]{lunardi} applies and shows that $L$ is sectorial.

Let us now compute the spectrum of the operator $L$. For this
purpose, we observe that, since $D(L)$ is compactly embedded into
$L^2$, $\sigma(L)$ consists of eigenvalues only. Further, if
$\lambda$ is an eigenvalue of $L$, then there exists a not
identically vanishing function $\psi$ such that $L\psi=\lambda\psi$.
In the Fourier variables, the previous equation leads to the
infinitely many equations
\begin{eqnarray*}
\lambda\widehat\psi(k)-l_k\widehat\psi(k)=0,\qquad\;\,k=0,1,2,\ldots
\end{eqnarray*}
If $\lambda\neq l_k$, then $\widehat\psi(k)=0$. Hence, if $\lambda$ is
not an element of the sequence $(l_k)$, $\lambda$ is in the
resolvent set of $L$. On the other hand, it is clear that the
sequence $(l_k)$ consists of eigenvalues of $L$. So
$\sigma(L)=\{l_k: k=0,1,\ldots\}$.

Since $l_k\to -\infty$ as $k\to +\infty$, $0$ is an isolated point
of the spectrum of $L$ and the corresponding eigenspace is one-dimensional. Let us prove that $\Pi$ is the spectral
projection associated with such an eigenvalue. For this purpose, we
prove that $0$ is a simple pole of the function $\lambda\mapsto
R(\lambda,L)$ and compute the residual at $0$. Note that for any
$\lambda\not\in\sigma(L)$ and any $\psi\in H^2_{\sharp}$ it holds
that
\begin{eqnarray*}
R(\lambda,L)\psi=\sum_{k=0}^{+\infty}\frac{1}{\lambda-l_k}\widehat \psi(k)w_k.
\end{eqnarray*}
Hence,
\begin{eqnarray*}
\lambda R(\lambda,L)\psi=\widehat \psi(0)w_0+\sum_{k=1}^{+\infty}\frac{\lambda}{\lambda-l_k}\widehat\psi(k)w_k=\Pi\psi
+\sum_{k=1}^{+\infty}\frac{\lambda}{\lambda-l_k}\widehat\psi(k)w_k.
\end{eqnarray*}
Hence, for
$|\lambda|\le\frac{1}{2}\displaystyle\min_{k=1,2,\ldots}|l_k|$, we
can estimate
\begin{eqnarray*}
|\lambda R(\lambda,L)\psi-\Pi\psi|_2^2\le
\sum_{k=1}^{+\infty}\left|\frac{\lambda}{\lambda-l_k}\right |^2|\widehat \psi(k)|^2\le
\frac{2|\lambda|^2}{l_{\min}}\sum_{k=1}^{+\infty}|\widehat\psi(k)|^2
\le \frac{2|\lambda|^2}{l_{\min}}|\psi|_2^2,
\end{eqnarray*}
where $l_{\min}=\min_{n=1,2,\ldots}|l_n|>0$.
This shows that $R(\lambda I-L)$ has a simple pole at $\lambda=0$ and
its residual is the operator $\Pi$, which turns out to be spectral
projection associated with the eigenvalue $0$, which is simple.
For more details, we
refer the reader to e.g., \cite[Prop. A.1.2 \& A.2.1]{lunardi}.

To conclude the proof of point (i), we observe that $l_k<0$, for
$k\ge 1$, if and only if $1+4\lambda_k-\alpha>0$. Since
$(\lambda_k)$ is a nondecreasing sequence, $l_k<0$ for any $k=1,2,\ldots$, if and only if
$4\lambda_1+1-\alpha>0$, i.e., if and only if $\alpha<\alpha_c$.

(ii), (iii) \& (iv). It is enough to observe that
$b_k\sim 4\lambda_k$, $f_k\sim 2\lambda_k^{3/2}$, $g_k\sim \frac{1}{2}\sqrt{\lambda_k}$ as $k\to +\infty$ and
$b_k\neq 0$ for any $k=0,1,\ldots$

(v). It is immediate and, hence, omitted.
\end{proof}

\subsection{Proof of Theorem $\ref{thm:1}$}
The proof is rather classical and is based on the results in
Propositions \ref{prop-symb-1}. Nevertheless, for the reader's
convenience we go into details. We split the proof in two steps: in
the first one we deal with Equation \eqref{second-order-varphi} and
in the second one we consider Equation \eqref{eq-4-order}.

{\em Step 1.} Using classical arguments based on a fixed point
argument, one can show that for any $\alpha\in\mathbb R$ and any
$T>0$, there exists $r_0>0$ such that, if
$\|\varphi_0\|_2\le r_0$, the Cauchy problem
\begin{equation}
\left\{
\begin{array}{lll}
\varphi_t(t,y)=(L\varphi(t,\cdot))(y)+(G((\varphi_y(t,\cdot))^2))(y), & t>0, & |y|\le\frac{\ell}{2},\\[1mm]
\varphi(t,-\ell/2)=\varphi(t,\ell/2), & t>0,\\[1mm]
\varphi_y(t,-\ell/2)=\varphi_y(t,\ell/2), & t>0,\\[1mm]
\varphi(0,y)=\varphi_0(y), &&|y|\le \frac{\ell}{2},
\end{array}
\right.
\label{pb-final}
\end{equation}
admits a unique solution $\psi\in \bigcup_{\theta\in (0,1)}{\mathscr
X}_{\theta}(T)$, where
\begin{eqnarray*}
{\mathscr X}_{\theta}(T)=\left\{\psi\in C([0,T];H^2_{\sharp})\cap C^1([0,T];L^2):
\sup_{0<\e<T}\e^{\theta}[\psi]_{C^{\theta}([\e,T];H^2_{\sharp})}\right\}.
\end{eqnarray*}
This can be proved slightly adapting the proof of \cite[Thm.
8.1.1]{lunardi}. The crucial point is the estimate
\begin{equation}
s^{\theta}|G((\psi_y(t,\cdot))^2)-G((\psi_y(s,\cdot))^2)|_2 \le
C_1s^{\theta}|\psi_y(s,\cdot)|_2[\psi]_{C^{\theta}([s,T];L^2)}|t-s|^{\theta},
\label{holder}
\end{equation}
for any $0<s<t\le T$, some positive constant $C_1$ and any $\psi\in {\mathscr
X}_{\theta}(T)$ ($\theta\in (0,1)$).
To prove this estimate it suffices to observe that, by Proposition \ref{prop-symb-1}(iv)
\begin{align*}
&|G((\psi_y(t,\cdot))^2)-G((\psi_y(s,\cdot))^2)|_2^2\\
\le &
C|(\psi_y(t,\cdot))^2-(\psi_y(s,\cdot))^2|_2
+C|D_y(\psi_y(t,\cdot))^2-D_y(\psi_y(s,\cdot))^2|_2^2\no\\
\le & C|\psi_y(t,\cdot)-\psi_y(s,\cdot)|_2\|\psi_y(t,\cdot)+\psi_y(s,\cdot)\|_{\infty}\\
&+C|\psi_{yy}(t,\cdot)|_2\|\psi_y(t,\cdot)-\psi_y(s,\cdot)\|_{\infty}\\
&+C\|\psi_y(s,\cdot)\|_{\infty}|\psi_{yy}(t,\cdot)-\psi_{yy}(s,\cdot)|_2\\
\le & C|\psi_y(t,\cdot)-\psi_y(s,\cdot)|_2|\psi_{yy}(t,\cdot)+\psi_{yy}(s,\cdot)|_2\\
&+C\left (|\psi_{yy}(t,\cdot)|_2+C|\psi_y(s,\cdot)|_2\right )|\psi_{yy}(t,\cdot)-\psi_{yy}(s,\cdot)|_2,
\end{align*}
for any $0<s<T$, where the last side of the previous chain of
inequalities follows from Poincar\'e-Wirtinger inequality, and $C$
denotes a positive constant, independent of $s$, $t$ and $\psi$,
which may vary from line to line. Estimate \eqref{holder} now
follows at once.

Let us now prove properties (a) and (b). It is convenient to split
the solution $\varphi$ to Equation \eqref{second-order-varphi} along
$\Pi (L^2)$ and $(I-\Pi)(L^2)$. We get $\varphi(t,y)=p(t)w_0+\psi(t,y)$
for any $t>0$ and any $y\in [-\ell/2,\ell/2]$. Since $\Pi$ commutes
with both the time and the spatial derivatives,
$\Pi(D_t\varphi)=D_t\Pi(\varphi)=p'$ and
$(I-\Pi)(D_y\varphi)=D_y(I-\Pi)(\varphi)=D_yw$. Moreover, for any
$\psi\in H^1_{\sharp}$,
$G(\psi)=\sum_{k=0}^{+\infty}g_k\widehat\psi(k)w_k$, so that
\begin{eqnarray*}
\Pi G(\psi)=g_0\widehat\psi(0)=-\frac{1}{2}\Pi\psi.
\end{eqnarray*}
Hence, projecting the Cauchy problem \eqref{pb-final} along
$\Pi(L^2)$ and $(I-\Pi)(L^2)$, we get the two self-consistent
equations for $p$ and $\psi$:
\begin{equation}
\left\{
\begin{array}{ll}
p'(t)=-\frac{1}{2}\Pi((\psi_y(t,\cdot))^2), & t>0,\\[1mm]
p(0)=\Pi(\varphi_0),
\end{array}
\right.
\label{pb-Pi}
\end{equation}
and
\begin{equation}
\left\{
\begin{array}{lll}
\psi_t(t,y)=(L\psi(t,\cdot))(y)+(I-\Pi)(G((\psi_y(t,\cdot))^2))(y), & t>0, & |y|\le\frac{\ell}{2},\\[1mm]
\psi(t,-\ell/2)=\psi(t,\ell/2), & t>0,\\[1mm]
\psi_y(t,-\ell/2)=\psi_y(t,\ell/2), & t>0,\\[1mm]
\psi(0,y)=((I-\Pi)(\varphi_0))(y), &&|y|\le \frac{\ell}{2}.
\end{array}
\right.
\label{pb-I-Pi}
\end{equation}
Clearly, the stability of the null solution to Equation
\eqref{second-order-varphi} depends only on the stability of the
null solution to the equation $\psi_t=L\psi+(I-\Pi)(G((\psi_y)^2))$, set in
$(I-\Pi)(L^2)$.

Note that the part of the operator $L$ in $(I-\Pi)(L^2)$ is still a
sectorial operator, and its spectrum is
$\sigma(L)\setminus\{0\}=\{l_k: k=1,2,\ldots\}$. In particular, all the elements of $\sigma(L)\setminus\{0\}$ lie
in $(-\infty,0)$. Hence, the
linearized stability principle applies to this situation. More
specifically,  in the case when $\alpha<\alpha_c$ all the
eigenvalues of the part of $L$ in $(I-\Pi)(L^2)$ are contained in
the plane $\{\lambda\in\mathbb C: {\rm Re}\,\lambda<0\}$. Hence, up
to replacing $r_0$ with a smaller value (if needed), for any
$\varphi\in B(0,r)\subset L^2$, the solution $\psi$ to Problem
\eqref{pb-I-Pi} exists for all the positive times. Moreover, for any
$\omega>\max\{l_k: k=1,2,\ldots\}$, there exists a positive constant
$C_{\omega}$ such that
\begin{eqnarray*}
|\psi_t(t,\cdot)|_2+\|\psi(t,\cdot)\|_2
\le C_{\omega}e^{\omega t}\|\varphi_0\|_2,\qquad\;\,t>0.
\end{eqnarray*}
As a byproduct, we can infer that the solution to Problem \eqref{pb-Pi} exists for all the positive times and
\begin{eqnarray*}
\lim_{t\to
+\infty}p(t)=p_{\infty}:=\Pi\varphi_0-\frac{1}{2}\int_0^{+\infty}\Pi((\psi_y(t,\cdot))^2)dt.
\end{eqnarray*}

Coming back to Problem \eqref{pb-final}, the above results show
that, if $\alpha<\alpha_c$, this problem admits a unique solution,
defined for all the positive times. Moreover,
\begin{eqnarray*}
|\varphi_t(t,\cdot)|_2+\|\varphi(t,\cdot)-p_{\infty}\|_{\infty}
+\|\varphi_y(t,\cdot)\|_{\infty} +|\varphi_{yy}(t,\cdot)|_2
\le P_{\omega}e^{\omega t}\|\varphi_0\|_2,
\end{eqnarray*}
for any $t>0$, any $\omega$ as above and some positive constant $P_{\omega}$ independent of $s$, $\varphi_0$ and $\varphi$, i.e., the null solution to Equation
\eqref{second-order-varphi} is (orbitally) stable with asymptotic
phase.

In the case when $\alpha>\alpha_c$ the spectrum of $L_{|(I-\Pi)(L^2)}$ contains
(a finite number of) eigenvalues with positive real part. Hence, the equation
$\psi_t=L\psi+(I-\Pi)(G((\psi_y)^2))$ admits a backward solution, exponentially decreasing to $0$ at $-\infty$
and this implies that the null solution to Problem \eqref{pb-I-Pi} and, consequently, the null solution to
Problem \eqref{pb-final} are unstable. For further details, we refer the reader to
e.g., \cite{henry} and \cite[Thm. 9.1.2 \& 9.1.3]{lunardi}.

{\em Step 2.}
We focus on the case when $\alpha<\alpha_c$, the other case being simpler.
Of course, we just need to deal with the function $\psi=(I-\Pi)\varphi$.
We assume that $\varphi_0\in H^4_{\sharp}$.
We are going to show that
for any $\omega\in (0,\max_{k=1,2,\ldots}l_k)$, it holds that
\begin{eqnarray*}
\sup_{t>0}e^{-\omega t}\|\varphi(t,\cdot)\|_4+
\sup_{t>0}e^{-\omega t}\|\varphi_t(t,\cdot)\|_2<+\infty.
\end{eqnarray*}

For this purpose, let us consider the differentiated problem
\begin{equation}
\left\{
\begin{array}{lll}
\rho_t(t,y)=(L\rho(t,\cdot))(y)+(D_{yy}(I-\Pi)(G(({\mathscr P}(\rho(t,\cdot)))^2))(y), & t>0, & |y|\le\frac{\ell}{2},\\[1mm]
\rho(t,-\ell/2)=\rho(t,\ell/2), & t>0,\\[1mm]
\rho_y(t,-\ell/2)=\rho_y(t,\ell/2), & t>0,\\[1mm]
\rho(0,y)=D_{yy}\varphi_0(y), &&|y|\le \frac{\ell}{2},
\end{array}
\right.
\label{pb-I-Pi-deriv}
\end{equation}
for the unknown $\rho=\psi_{yy}$.
Here,
\begin{equation}
{\mathscr P}(\zeta)=(I-\Pi)\left (y\mapsto \int_{-\frac{\ell}{2}}^y\zeta(\eta)d\eta\right ),\qquad\;\,\zeta\in L^2.
\label{operat-P}
\end{equation}

This problem has the same structure as Problem \eqref{pb-I-Pi}, and, by assumptions,
$D_{yy}\varphi_0\in H^2_{\sharp}$. Therefore, up to taking a smaller
$r_0$ (if necessary), if $\|D_{yy}\varphi_0\|_2\le r_0$, Problem
\eqref{pb-I-Pi-deriv} has a solution $\rho$ which belongs to
$C^1([0,T];L^2)\cap C([0,T];H^2_{\sharp})$ for any $T>0$. Moreover,
\begin{eqnarray*}
\sup_{t>0}e^{-\omega t}\|\rho(t,\cdot)\|_2<+\infty,
\end{eqnarray*}
and $\rho(t,\cdot)=(I-\Pi)\rho(t,\cdot)$ for any $t>0$.

Let us show that $\psi(t,\cdot)={\mathscr P}^2(\rho(t,\cdot))$ for any $t>0$.
Clearly, the function $\Psi={\mathscr P}^2(\rho)$ belongs to $C([0,+\infty);H^4_{\sharp})\cap
C^1([0,+\infty);L^2)$. Moreover, it belongs to ${\mathscr X}_{1/2}(T)$ for any $T>0$.
Indeed, $H^2_{\sharp}$ belongs to the class $J_{1/2}$ between $L^2$ and $H^4_{\sharp}$.
This means that
\begin{align*}
\|\Psi(t,\cdot)-\Psi(s,\cdot)\|_2&\le C|\Psi(t,\cdot)-\Psi(s,\cdot)|_2^{\frac{1}{2}}
\|\Psi(t,\cdot)-\Psi(s,\cdot)\|_4^{\frac{1}{2}}\\
\le& \sqrt{2}C\|\Psi_t\|_{C([0,T];L^2)}^{\frac{1}{2}}\|\Psi\|_{C([0,T];H^4_{\sharp})}^{\frac{1}{2}}
|t-s|^{\frac{1}{2}},
\end{align*}
for any $0\le s\le t\le T$ and some positive constant $C$, independent of $s$, $t$ and $\Psi$.
From this estimate, it is clear that $\Psi\in C^{1/2}([0,T];H^2_{\sharp})\subset {\mathscr X}_{1/2}(T)$
for any $T>0$.

Further, $D_{yy}\Psi=\rho$ and $D_t\Psi={\mathscr P}^2(D_t\rho)$,
so that $D_{yy}D_t\Psi=D_tD_{yy}\Psi=\rho_t$.
It turns out that
\begin{eqnarray*}
(i)~D_{yy}(D_t\Psi-L\Psi-(I-\Pi)G((\Psi_y)^2))\equiv 0,\qquad\;\, (ii)~\Psi(0,\cdot)\equiv (I-\Pi)\varphi_0.
\end{eqnarray*}
Hence, $D_t\Psi-L\Psi-(I-\Pi)G((\Psi_y)^2)=a(t)+b(t)y$ for some functions $a,b:[0,+\infty)\to\mathbb R$.
Since $\Psi_t$, $L\Psi$ and $G((\Psi_y)^2)$ are continuous functions
in $[0,+\infty)\times [-\ell_0/2,\ell_0/2]$ and are periodic with respect to $y$, it follows that
$D_t\Psi-L\Psi-(I-\Pi)G((\Psi_y)^2)$ is periodic with respect to $y$ as well. Moreover,
this latter function belongs to $(I-\Pi)(L^2)$ since $\Psi$ does.
Hence, $a=b\equiv 0$, implying that $\Psi$ and $\psi$ actually coincide. We have so proved that $\psi\in C([0,+\infty);H^4_{\sharp})$ and $D_t\psi\in C([0,+\infty);H^2_{\sharp})$. Moreover,
\begin{eqnarray*}
\sup_{t>0}e^{-\omega t}|\psi_{tyy}(t,\cdot)|_2+\sup_{t>0}e^{-\omega t}\|\psi(t,\cdot)\|_{H^4_{\sharp}}<+\infty.
\end{eqnarray*}

To complete the proof it suffices to show that $\varphi$ solves
Equation \eqref{eq-4-order}, but this follows immediately observing
that $\varphi$ is in the domain of both the operators $B$ (see
Proposition \ref{prop-symb-1}(ii)) and $S$, and $(\varphi_y)^2$ is
in the domain of the operator $F$ (see Proposition
\ref{prop-symb-1}(iii)). Further, $L=B^{-1}S$ and $G=B^{-1}F$ in $H^4_{\sharp}$.
Since $\varphi$ solves the differential
equation $\varphi_t-L\varphi-G((\varphi_y)^2)=0$, applying $B$ to both the sides of the equation,
it now follows immediately that $\varphi$ solves Equation \eqref{eq-4-order}.

\section{Rigorous derivation of the Kuramoto-Sivashinsky equation}
\setcounter{equation}{0} \label{sect-5} In this section we are
interested in proving Theorem \ref{thm:2}.
\subsection{Rescaling and equation for the remainder}
Let $\varphi$ be a solution to \eqref{psd-abstract}. We set $\alpha=1+\e$ and define the rescaled dependent and
independent variables:
\begin{eqnarray*}
t= \tau/\varepsilon^2,\quad y = \eta/\sqrt{\varepsilon}, \quad
\varphi = \varepsilon \psi.
\end{eqnarray*}
The spatial period is now $\ell_{\e} = \ell_0/{\sqrt{\e}}$, for some
$\ell_0> 4\pi$ fixed, see the Introduction. A straightforward
computation reveals that the function $\psi$ satisfies the equation
\begin{equation}
\frac{\partial}{\partial\tau} {\mathscr B}_{\e}(\psi)={\mathscr
S}(\psi)+{\mathscr F}_{\e}((D_{\eta}\psi)^2),
 \label{eq-4th-order-e-bis}
\end{equation}
where
\begin{align*}
&{\mathscr B}_{\e}=I-4\e D_{\eta\eta}+(1+\e)\left (\sqrt{I-4\e
D_{\eta\eta}}-1\right ),\\
&{\mathscr S}=-4D_{\eta\eta\eta\eta}-D_{\eta\eta},\\
&{\mathscr F}_{\e}= \frac{1}{4}\left\{(I-4\e
D_{\eta\eta})^{\frac{3}{2}}-3(I-4\e D_{\eta\eta})-4(1+\e)\left (\sqrt{1-4\e
D_{\eta\eta}}-I\right ) \right\}(D_{\eta}\cdot)^2.
\end{align*}

Note that, if we denote by $(\lambda_k(\ell))$ the sequence of the
eigenvalues of the second-order derivative with periodic boundary
conditions in $[-\ell/2,\ell/2]$, it turns out that
$\lambda_k(\ell_{\e})=\e\lambda_k(\ell_0):=\e\lambda_k$, for any
$k=0,1,\ldots$. Hence, the symbols of the operators ${\mathscr
B}_{\e}$, ${\mathscr S}$ and ${\mathscr F}_{\e}$ are
\begin{align}
&b_{\e,k}=X^2_{\e,k}+(1+\e) X_{\e,k}-1-\e,\notag\\
&s_k=-\lambda_k(4\lambda_k-1)\notag,\\
&f_{\e,k}=\frac{1}{4}(X^3_{\e,k}-3X_{\e,k}^2-4(1+\e) X_{\e,k}+4+4\e),
\label{fk}
\end{align}
for any $k=0,1,\ldots$, where
\begin{eqnarray*}
X_{\varepsilon,k}=\sqrt{1+4\varepsilon\lambda_k},\qquad\;\,k=0,1,\ldots
\end{eqnarray*}
Hence, the equation for the function $\psi$ (in Fourier
coordinates) reads
\begin{align*}
&b_{\e,k}\widehat\psi_{\tau}(\tau,k)=-\lambda_k(4\lambda_k-1)\widehat\psi(\tau,k)+f_{\e,k}\widehat{(\psi_{\eta})^2}(\tau,k),
\end{align*}
for any $k=0,1,\ldots$ Note that the leading terms (at order $0$ in $\e$) of $b_{\e,k}$ and $f_{\e,k}$ are
$1$ and $-1/2$, respectively. Hence, at the zero-order, we recover the K--S
equation
\begin{eqnarray*}
\Phi_{\tau}+4\Phi_{\eta\eta\eta\eta}+\Phi_{\eta\eta}+\frac{1}{2}(\Phi_{\eta})^2=0.
\end{eqnarray*}
As we remind it in the Introduction, this equation has been thoroughly studied by many authors.
For our purposes, we need the following classical result.
For the reader's convenience we provide a rather detailed proof in Appendix \ref{sect-appendix}.
\begin{theorem}
\label{thm-4.1} Let $\Phi_0\in H^m_{\sharp}$ for some $m\ge 4$ and fix $T>0$.
Then, the Cauchy problem
\begin{equation}
\left\{
\begin{array}{lll}
\Phi_\tau(\tau,\eta)= -4\Phi_{\eta\eta\eta\eta}(\tau,\eta)-
\Phi_{\eta\eta}(\tau,\eta)
- \frac{1}{2}(\Phi_\eta(\tau,\eta))^2, &\tau\ge 0, & |\eta|\le\frac{\ell_0}{2},\\[2mm]
D_{\eta}^k\Phi(\tau,-\ell_0/2)=D_{\eta}^k\Phi(\tau,\ell_0/2), &\tau\ge 0, & k=0,1,2,3,\\[2mm]
\Phi(0,\eta)=\Phi_0(\eta), &&|\eta|\le \frac{\ell_0}{2},
\end{array}
\right.
\label{4order}
\end{equation}
admits a unique solution $\Phi\in C([0,T];H^m_{\sharp})$ such that
$\Phi_{\tau}\in C([0,T];H^{m-4}_{\sharp})$.
\end{theorem}

The above (heuristical) arguments suggest to split $\psi$ as
follows:
\begin{eqnarray*}
\psi= \Phi+ \e\rho_\e.
\end{eqnarray*}
To avoid cumbersome
notation, we simply write $\rho$ for $\rho_\e$, when there is no damage of confusion.
By assumptions (see Theorem \ref{thm:2}), the initial condition for $\rho$ is
\begin{align*}
\rho(0,\cdot)=0.
\end{align*}
Replacing into \eqref{eq-4th-order-e-bis} we get, after simplifying by
$\e$,
\begin{align}
\frac{\partial}{\partial\tau}{\mathscr B}_{\e}(\rho) +{\mathscr
H}_{\e}(\Phi_{\tau}) ={\mathscr S}(\rho)+{\mathscr
M}_{\e}((\Phi_{\eta})^2)+\e{\mathscr
F}_{\e}((\rho_{\eta})^2)+2{\mathscr F}_{\e}(\Phi_{\eta}\rho_{\eta}),
\label{pb-fin-ve}
\end{align}
for any $k=0,1,\ldots$, where the symbols of the operators
${\mathscr H}_{\e}$ and ${\mathscr M}_{\e}$ are
\begin{align}
h_{\e,k}&=\frac{1}{\e}(X^2_{\e,k}+(1+\e) X_{\e,k}-2-\e), \label{hk}
\\[1mm]
m_{\e,k}&=\frac{1}{4\e}(X^3_{\e,k}-3X_{\e,k}^2-4(1+\e)
X_{\e,k}+6+4\e), \label{mk}
\end{align}
for any $k=0,1,\ldots$

\begin{proposition}
\label{prop-H} Fix $\e\in (0,1]$. Then there exists a positive
constant $C_*$ such that the following properties are met:
\begin{enumerate}[\rm (a)]
\item
for any $s=2,3,\ldots$, the operators ${\mathscr B}_{\e}$ and
${\mathscr H}_{\e}$ admit bounded realizations $B_{\e}$ and
$H_{\e}$, respectively, mapping $H^s_{\sharp}$ into
$H^{s-2}_{\sharp}$. Moreover
\begin{eqnarray*}
\|B_{\e}\|_{L(H^s_{\sharp},H^{s-2}_{\sharp})}
+\|H_{\e}\|_{L(H^s_{\sharp},H^{s-2}_{\sharp})}\le C_*,
\end{eqnarray*}
for any $\e\in (0,1]$ and any $s$ as above. Finally, the operator
$B_{\e}$ is invertible both from $H^s_{\sharp}$ to
$H^{s-2}_{\sharp}$ for any $s=2,3,\ldots$;
\item
for any $s\ge 3$, the operators ${\mathscr F}_{\e}$ and ${\mathscr
M}_{\e}$ admit bounded realizations $F_{\e}$ and $M_{\e}$, respectively, mapping $H^s$ into $H^{s-3}$. Moreover,
\begin{eqnarray*}
\|F_{\e}\|_{L(H^s_{\sharp},H^{s-3}_{\sharp})}+\|M_{\e}\|_{L(H^s_{\sharp},H^{s-3}_{\sharp})}\le C_*,
\end{eqnarray*}
for any $\e\in (0,1]$ and any $s=3,4,\ldots$
\end{enumerate}
\end{proposition}

\begin{proof}
(a). A straightforward computation shows that
\begin{align*}
|h_{\e,k}|= 4\lambda_k+\frac{4(\e+1)\lambda_k}{\sqrt{1+4\e\lambda_k}+1}
\le 4\lambda_k+2(\e+1)\lambda_k=(6+2\e)\lambda_k,
\end{align*}
for any $k=0,1,\ldots$ and any $\e\in (0,1]$. This shows that ${\mathscr H}_{\e}$ admits a
bounded realization mapping $H^s_{\sharp}$ into $H^{s-2}_{\sharp}$
for any $s\ge 2$ and its norm can be bounded by a constant,
independent of $\e\in (0,1]$.

Since $b_{\e,k}=\e h_{\e,k}+1$ for any $k=0,1,\ldots$, the
boundedness of the operator $B_{\e}$ from $H^s_{\sharp}$ to
$H^{s-2}_{\sharp}$ follows at once.

Showing that the operator $B_{\e}$ is invertible from $H^s_{\sharp}$
into $H^{s-2}_{\sharp}$ is an easy task. It suffices to observe that
$b_{\e,k}\ge 4\e\lambda_k+1$ for any $k=0,1,\ldots$.

(b). Since $f_k=\e m_{\e,k}-1/2$ for any
$k=0,1,\ldots$, we can limit ourselves to considering the operator
${\mathscr M}_{\e}$. A simple computation shows that
\begin{align*}
|m_{\e,k}|&\le \frac{(1+4\e\lambda_k)^{\frac{3}{2}}-1}{4\e}+3\lambda_k+4\frac{1+\e}{\e}\left (\sqrt{1+4\e\lambda_k}-1\right )\\
&\le\frac{16\e^2\lambda_k^3+12\e\lambda_k^2+3\lambda_k}{(1+4\e\lambda_k)^{\frac{3}{2}}+1}
+19\lambda_k\\
&\le\frac{16\e^2\lambda_k^3}{(1+4\e\lambda_k)^{\frac{3}{2}}}
+\frac{12\e\lambda_k^2}{1+4\e\lambda_k}
+22\lambda_k\\
&\le 2\sqrt{\e}\lambda_k^{\frac{3}{2}}
+25\lambda_k,
\end{align*}
for any $k=0,1,2,\ldots$ Hence, $M_{\e}$ is well defined (and
bounded) in $H^s_{\sharp}$ with values in $H^{s-3}_{\sharp}$ for any
$s\ge 3$. Since its symbol can be estimated from above uniformly with
respect to $\e\in (0,1]$, the assertion follows immediately.
\end{proof}

Since all the operators appearing in \eqref{pb-fin-ve} commute with
$D_{\eta}$, the {\it differentiated} problem for $\zeta:=\rho_{\eta}$ reads as follows:
\begin{equation}
\frac{\partial}{\partial\tau}{\mathscr B}_{\e}(\zeta) +{\mathscr H}_{\e}(\Psi_{\tau}) ={\mathscr S}(\zeta)+{\mathscr M}_{\e}((\Psi^2)_{\eta})+\e {\mathscr F}_{\e}((\zeta^2)_{\eta})+2{\mathscr F}_{\e}((\Psi\zeta)_{\eta}),
\label{eq-final-ve-1}
\end{equation}
where we have set $\Psi=\Phi_{\eta}$. Obviously, it has a null
initial condition at time $\tau=0$. For simplicity, we denote
$D_{\eta}$ by $D$. For an integer $n \geq 1$, $D^n$ is the
differentiation operator of order $n$. We also set $D^0={\rm Id}$.

\subsection{Formal a priori estimates}\label{Formal-estimates}

For any $n=0,1,2,\ldots$ and any $T>0$, we set
\begin{eqnarray*}
{\mathscr Y}_n(T)=\left\{\zeta\in C([0,T];H^{4\vee 2n}_{\sharp})\cap C^1([0,T];L^2):\zeta_{\tau}\in C([0,T];H^{2\vee (n+1)}_{\sharp})\right\},
\end{eqnarray*}
where $a\vee b:=\max\{a,b\}$.
The main result of this subsection is contained in the following theorem.

\begin{theorem}
\label{cor-apriori-estim} Fix an integer $n\geq 0$ and $T>0$.
Further, fix $m$ in Theorem $\ref{thm-4.1}$ large enough such that
$\Psi\in C([0,T];H^{n+4}_{\sharp})\cap C^1([0,T];L^2)$ and $\Psi_{\tau}\in C([0,T];H^{n+2}_{\sharp})$.
Then, there exist
$\varepsilon_1=\e_1(n,T)\in (0,1)$ and $K_n=K_n(n,T)>0$ such that, if
$\zeta\in {\mathscr Y}_n(T_1)$ is a solution on the time interval $[0,T_1]$ of
Equation \eqref{eq-final-ve-1} for some $T_1\le T$, then
\begin{align}\label{main-estimate}
\sup_{\tau\in [0,T_1]}
\int_{-\frac{\ell_0}{2}}^{\frac{\ell_0}{2}}&|D^{n}\zeta(\tau,\cdot)|^2d\eta
+\e\sup_{\tau\in [0,T_1]}
\int_{-\frac{\ell_0}{2}}^{\frac{\ell_0}{2}}|D^{n+1}\zeta(\tau,\cdot)|^2d\eta
\leq K_n,
\end{align}
whenever $0<\varepsilon \leq\varepsilon_1$.
\end{theorem}

To prove \eqref{main-estimate}, we multiply both sides of the equation \eqref{eq-final-ve-1} by
$(-1)^n D^{2n}\zeta$ and integrate by parts over
$(-\ell_0/2,\ell_0/2)$. We thus get
\begin{align}
&(-1)^n\int_{-\frac{\ell_0}{2}}^{\frac{\ell_0}{2}}B_{\e}(\zeta_{\tau}(\tau,\cdot))D^{2n}\zeta(\tau,\cdot)d\eta
-(-1)^n\int_{-\frac{\ell_0}{2}}^{\frac{\ell_0}{2}}S(\zeta(\tau,\cdot))D^{2n}\zeta(\tau,\cdot)d\eta \notag \\
=&  -(-1)^n\int_{-\frac{\ell_0}{2}}^{\frac{\ell_0}{2}}\left
(H_{\e}(\Psi_{\tau}(\tau,\cdot))
-M_{\e}((\Psi^2)_{\eta}(\tau,\cdot))\right ) D^{2n}\zeta(\tau,\cdot)d\eta\notag\\
&+(-1)^n\e\int_{-\frac{\ell_0}{2}}^{\frac{\ell_0}{2}}F_{\e}((\zeta^2)_{\eta}(\tau,\cdot))
D^{2n}\zeta(\tau,\cdot)d\eta \notag\\
&+2(-1)^n\int_{-\frac{\ell_0}{2}}^{\frac{\ell_0}{2}}F_{\e}((\Psi\zeta)_{\eta}(\tau,\cdot))
D^{2n}\zeta(\tau,\cdot)d\eta,
\label{variational}
\end{align}
where $S=-4D_{\eta\eta\eta\eta}-D_{\eta\eta}$.

In the following lemmata we estimate all the terms appearing in the
previous equation. We first deal with the left-hand side of
\eqref{variational} which consists of the ``benign'' terms.
\begin{lemma}
\label{lemma-1}
Fix $n=0,1,\ldots$, $\e>0$, $T_1\le T$ and $\zeta\in {\mathscr Y}_n(T_1)$. Then,
\begin{align}
&(-1)^n\int_{-\frac{\ell_0}{2}}^{\frac{\ell_0}{2}}B_{\e}(\zeta_{\tau}(\tau,\cdot))D^{2n}\zeta(\tau,\cdot)d\eta
- (-1)^n\int_{-\frac{\ell_0}{2}}^{\frac{\ell_0}{2}}S(\zeta(\tau,\cdot))D^{2n}\zeta(\tau,\cdot)d\eta\notag\\
=&\frac{1}{2}\frac{d}{d\tau}\int_{-\frac{\ell_0}{2}}^{\frac{\ell_0}{2}}|D^{n}\zeta(\tau,\cdot)|^2d\eta
+2\e\frac{d}{d\tau}\int_{-\frac{\ell_0}{2}}^{\frac{\ell_0}{2}}|D^{n+1}\zeta(\tau,\cdot)|^2d\eta\notag\\
&+\frac{(1+\e)}{2}\frac{d}{d\tau}\int_{-\frac{\ell_0}{2}}^{\frac{\ell_0}{2}}
\langle R_{\e}D^{n}\zeta(\tau,\cdot),D^{n}\zeta(\tau,\cdot)\rangle d\eta\notag\\
&+4\int_{\frac{\ell_0}{2}}^{\frac{\ell_0}{2}}|D^{n+2}\zeta(\tau,\cdot)|^2d\eta-
\int_{\frac{\ell_0}{2}}^{\frac{\ell_0}{2}}|D^{n+1}\zeta(\tau,\cdot)|^2d\eta,
\label{estim-1}
\end{align}
where $R_{\e}:H^1_{\sharp}\to L^2$ is the positive
operator whose symbol is $(X_{\e,k}-1)$.
\end{lemma}
\begin{proof}
For any $\zeta\in {\mathscr Y}_n(T_1)$, we can estimate
\begin{align}
&\int_{-\frac{\ell_0}{2}}^{\frac{\ell_0}{2}}B_{\e}(\zeta_{\tau}(\tau,\cdot))D^{2n}\zeta(\tau,\cdot)
d\eta
= (-1)^n\sum_{k=1}^{+\infty}b_{\e,k}\widehat\zeta_{\tau}(\tau,k)\lambda_k^n\widehat\zeta(\tau,k)\notag\\
=&(-1)^n\sum_{k=1}^{+\infty}\lambda_k^n\widehat\zeta_{\tau}(\tau,k)\widehat\zeta(\tau,k)
+4(-1)^n \e\sum_{k=1}^{+\infty}\lambda_k^{n+1}\widehat\zeta_{\tau}(\tau,k)\widehat\zeta(\tau,k)\notag\\
&+ (-1)^n  (1+\e)\sum_{k=1}^{+\infty}\lambda_k^n (X_{\e,k}-1)\widehat\zeta_{\tau}(\tau,k)\widehat\zeta(\tau,k)\notag\\
=&\int_{-\frac{\ell_0}{2}}^{\frac{\ell_0}{2}}\zeta_{\tau}(\tau,\cdot)D^{2n}\zeta(\tau,\cdot)d\eta
-4\e\int_{-\frac{\ell_0}{2}}^{\frac{\ell_0}{2}}\zeta_{\tau\eta\eta}(\tau,\cdot) D^{2n}\zeta(\tau,\cdot)d\eta\notag\\
&+(1+\e)\int_{-\frac{\ell_0}{2}}^{\frac{\ell_0}{2}}\zeta_{\tau}(\tau,\cdot)(\sqrt{I-4\e
D_{\eta\eta}}-I)D^{2n}\zeta(\tau,\cdot)
d\eta\notag\\
=&(-1)^n\int_{-\frac{\ell_0}{2}}^{\frac{\ell_0}{2}}(D^{n}\zeta)_{\tau}(\tau,\cdot)D^{n}\zeta(\tau,\cdot)d\eta
+4(-1)^n\e\int_{-\frac{\ell_0}{2}}^{\frac{\ell_0}{2}}D^{n+1}(\zeta_{\tau}(\tau,\cdot))D^{n+1}\zeta(\tau,\cdot)d\eta\notag\\
&+(-1)^n(1+\e)\int_{-\frac{\ell_0}{2}}^{\frac{\ell_0}{2}}\langle
R_{\e}D^n\zeta(\tau,\cdot),
(D^n\zeta)_{\tau}(\tau,\cdot)\rangle d\eta\notag\\
&-(-1)^n(1+\e)\int_{-\frac{\ell_0}{2}}^{\frac{\ell_0}{2}}(D^n\zeta)_{\tau}(\tau,\cdot)D^n\zeta(\tau,\cdot)d\eta\notag\\
=&(-1)^n\bigg\{\frac{1}{2}\frac{d}{d\tau}\int_{-\frac{\ell_0}{2}}^{\frac{\ell_0}{2}}|D^n\zeta(\tau,\cdot)|^2d\eta
+2\e\frac{d}{d\tau}\int_{-\frac{\ell_0}{2}}^{\frac{\ell_0}{2}}|D^{n+1}\zeta(\tau,\cdot)|^2d\eta\notag\\
&\qquad\qquad+\frac{(1+\e)}{2}\frac{d}{d\tau}\int_{-\frac{\ell_0}{2}}^{\frac{\ell_0}{2}}
\langle R_{\e}D^n \zeta(\tau,\cdot),D^n \zeta(\tau,\cdot)\rangle d\eta \bigg\},\notag\\
\label{estim-2}
\end{align}
for any $\tau\in [0,T_1]$. On the other hand, a straightforward computation shows that
\begin{align}
\int_{-\frac{\ell_0}{2}}^{\frac{\ell_0}{2}}S(\zeta(\tau,\cdot))\,D^{2n}\zeta(\tau,\cdot)d\eta=&
-4(-1)^n\int_{-\frac{\ell_0}{2}}^{\frac{\ell_0}{2}}|D^{n+2}\zeta(\tau,\cdot)|^2d\eta
\label{estim-3}\\
&+(-1)^n\int_{-\frac{\ell_0}{2}}^{\frac{\ell_0}{2}}|D^{n+1}\zeta(\tau,\cdot)|^2d\eta,
\notag
\end{align}
for any $\tau\in [0,T_1]$.
Combining \eqref{estim-2} and \eqref{estim-3}, Estimate \eqref{estim-1} follows at once.
\end{proof}

We now deal with the other terms in \eqref{variational}.

\begin{lemma}
\label{lemma-2}
Fix $n=0,1,\ldots$, $T_1\le T$ and assume that $\Psi\in C([0,T_1];H^{n+4}_{\sharp})$ and
$\Psi_{\tau}\in C([0,T_1];H^{n+2}_{\sharp})$. Then, there exist a positive constant $C_n$, independent of $\e\in (0,1]$ and $T_1$, and constants
$K_1'(n,\Psi)$ and $K_2'(n,\Psi)$ such that the following estimates hold
\begin{align}
&\left |\int_{-\frac{\ell_0}{2}}^{\frac{\ell_0}{2}}\left
(H_{\e}(\Psi_{\tau}(\tau,\eta))-M_{\e}((\Psi^2)_{\eta}(\tau,\cdot))\right
)D^{2n}\zeta(\tau,\cdot)d\eta\right | \leq K_1'(n,\Psi)+|D^n
\zeta(\tau,\cdot)|_2^2; \label{estim-bad-1}
\\[2mm]
&\left
|\int_{-\frac{\ell_0}{2}}^{\frac{\ell_0}{2}}F_{\e}((\zeta^2)_{\eta})D^{2n}\zeta(\tau,\cdot)d\eta\right
| \le
C_n\e^{\frac{3}{2}}|D^n\zeta(\tau,\cdot)|_2|D^{n+2}\zeta(\tau,\cdot)|_2^2
+C_n|D^n\zeta(\tau,\cdot)|_2^4
\label{estim-bad-2}\\
&\qquad\qquad\qquad\qquad\qquad\qquad\qquad+C_n\e^2|D^{n+1}\zeta(\tau,\cdot)|_2^4
+C_n|D^{n+2}\zeta(\tau,\cdot)|_2^2;\notag\\[1mm]
&\left|\int_{-\frac{\ell_0}{2}}^{\frac{\ell_0}{2}}F_{\e}((\Psi\zeta)_{\eta}(\tau,\cdot))
D^{2n}\zeta(\tau,\cdot)d\eta\right |
\label{estim-bad-3}\\
\le &K_2'(n,\Psi)\left (\e|D^{n+2}\zeta(\tau,\cdot)|_2^2+\e|D^{n+1}\zeta(\tau,\cdot)|_2^2
+|D^n\zeta(\tau,\cdot)|_2^2\right )
+\frac{1}{4}|D^{n+2}\zeta(\tau,\cdot)|_2^2,\notag
\end{align}
for any $\tau\in [0,T_1]$ and any $\zeta\in {\mathscr Y}_n(T_1)$.
\end{lemma}

\begin{proof}
Fix $n=0,1,\ldots$
Throughout the proof $C$ denotes a positive constant depending on $n$, but being independent of $\tau$, $\Psi$ and $\zeta$,
which may vary from line to line.

Estimate \eqref{estim-bad-1} follows immediately from Proposition \ref{prop-H}, Poincar\'e-Wirtinger and Cauchy-Schwarz
inequalities, which allow us to estimate
\begin{align*}
&\left |\int_{-\frac{\ell_0}{2}}^{\frac{\ell_0}{2}}\left
(H_{\e}(\Psi_{\tau}(\tau,\cdot))
-M_{\e}((\Psi^2)_{\eta}(\tau,\cdot))\right )D^{2n}\zeta(\tau,\cdot)d\eta\right |\\
= &\left |\int_{-\frac{\ell_0}{2}}^{\frac{\ell_0}{2}}D^n\left
(H_{\e}(\Psi_{\tau}(\tau,\cdot))
-M_{\e}((\Psi^2)_{\eta}(\tau,\cdot))\right )D^{n}\zeta(\tau,\cdot)d\eta\right |\\
\le & K_1'(n,\Psi)
+|D^n\zeta(\tau,\cdot)|_2^2,
\end{align*}
for any $\tau\in [0,T_1]$.

Let us now prove Estimate \eqref{estim-bad-2}. For this purpose, we observe that
\begin{eqnarray*}
|f_{\e,k}|\le
2\sqrt{2}\e^{\frac{3}{2}}\lambda_k^{\frac{3}{2}}+3\e\lambda_k+2(\e+1)\sqrt{\e}\lambda_k^{\frac{1}{2}}+\frac{3+\sqrt{2}}{4},\qquad\;\,k=0,1,\ldots
\end{eqnarray*}
For the convenience of the reader, we explicit the splittings we use below:
\begin{eqnarray*}
&\displaystyle\frac{3}{2} +n = \frac{2+n}{2} + \frac{1+n}{2},\qquad  & 1+n= \frac{1+n}{2} + \frac{1+n}{2},\\[1mm]
&\displaystyle\frac{1}{2} +n= \frac{n-1}{2} +\frac{2+n}{2},\qquad & n=\frac{n-1}{2}+ \frac{n+1}{2},
\end{eqnarray*}
for $n\ge 1$. (The case $n=0$ can be handled likewise with very few slight and straightforward changes.)
Hence, for any $\chi\in C([0,T_1];H^{4\vee 2n}_{\sharp})$ we can estimate
\begin{align}
&\left
|\int_{-\frac{\ell_0}{2}}^{\frac{\ell_0}{2}}F_{\e}(\chi_{\eta})D^{2n}\zeta
d\eta\right |
=\sum_{k=0}^{+\infty}\lambda_k^n|n_{\e,k}|\widehat{\chi_{\eta}}(\tau,k)|\widehat\zeta(\tau,k)|
\label{estim-6}\\
\le &2\sqrt{2}\e^{\frac{3}{2}}\sum_{k=0}^{+\infty}\lambda_k^{\frac{3}{2}+n}|\widehat{\chi_{\eta}}(\tau,k)||\widehat\zeta(\tau,k)|
+3\e\sum_{k=0}^{+\infty}\lambda_k^{1+n}|\widehat{\chi_{\eta}}(\tau,k)||\widehat\zeta(\tau,k)|\notag\\
&+2\sqrt{\e}(1+\e)\sum_{k=0}^{+\infty}\lambda_k^{\frac{1}{2}+n}|\widehat{\chi_{\eta}}(\tau,k)||\widehat\zeta(\tau,k)|\notag\\
&+\frac{3+\sqrt{2}}{4}\sum_{k=0}^{+\infty}\lambda_k^n|\widehat{\chi_{\eta}}(\tau,k)||\widehat\zeta(\tau,k)|\notag\\
\le&2\sqrt{2}\e^{\frac{3}{2}}\left (\sum_{k=0}^{+\infty}\lambda_k^{1+n}|\widehat{\chi_{\eta}}(\tau,k)|^2\right )^{\frac{1}{2}}
\left (\sum_{k=0}^{+\infty}\lambda_k^{2+n}|\widehat\zeta(\tau,k)|^2\right )^{\frac{1}{2}}\notag\\
&+3\e\left (\sum_{k=0}^{+\infty}\lambda_k^{1+n}|\widehat{\chi_{\eta}}(\tau,k)|^2\right )^{\frac{1}{2}}
\left (\sum_{k=0}^{+\infty}\lambda_k^{1+n}|\widehat\zeta(\tau,k)|^2\right )^{\frac{1}{2}}\notag\\
&+2\sqrt{\e}(1+\e)\left (\sum_{k=0}^{+\infty}|\lambda_k^{n-1}\widehat{\chi_{\eta}}(\tau,k)|^2\right )^{\frac{1}{2}}
\left (\sum_{k=0}^{+\infty}\lambda_k^{2+n}|\widehat\zeta(\tau,k)|^2\right )^{\frac{1}{2}}\notag\\
&+\frac{3+\sqrt{2}}{4}\left (\sum_{k=0}^{+\infty}|\lambda_k^{n-1}\widehat{\chi_{\eta}}(\tau,k)|^2\right )^{\frac{1}{2}}
\left (\sum_{k=0}^{+\infty}\lambda_k^{n+1}|\widehat\zeta(\tau,k)|^2\right )^{\frac{1}{2}}\notag\\
=&2\sqrt{2}\e^{\frac{3}{2}}|D^{n+2}\chi(\tau,\cdot)|_2|D^{n+2} \zeta(\tau,\cdot)|_2
+3\e|D^{n+2}\chi(\tau,\cdot)|_2|D^{n+1}\zeta(\tau,\cdot)|_2\notag\\
&+2\sqrt{\e}(1+\e)|D^n\chi(\tau,\cdot)|_2|D^{n+2}\zeta(\tau,\cdot)|_2\notag\\
&+\frac{3+\sqrt{2}}{4}|D^n\chi(\tau,\cdot)|_2|D^{n+1}\zeta(\tau,\cdot)|_2,\notag
\end{align}
for any $\tau\in [0,T_1]$.

Now, we are in a position to prove Estimate \eqref{estim-bad-2}. For this purpose, we observe that,
using the Leibniz formula and the Poincar\'e-Wirtinger inequality, it comes:
\begin{align}
&|D^{n+2}(\zeta(\tau,\cdot))^2|_2 \leq C (|D^{n+2}\zeta(\tau,\cdot)|_2 |D^{n}\zeta(\tau,\cdot)|_2 + |D^{n+1}\zeta(\tau,\cdot)|_2 ^2)\label{estim-7}
\\
&|D^{n}(\zeta(\tau,\cdot))^2|_2 \leq C |D^{n}\zeta(\tau,\cdot)|_2^2,
\label{estim-8}
\end{align}
for any $\tau\in [0,T_1]$.
Replacing Estimates \eqref{estim-7}, \eqref{estim-8} in \eqref{estim-6} (with $\chi=\zeta^2$), and using the Cauchy-Schwarz inequality
and, again, the Poincar\'e-Wirtinger inequality, we get
\begin{align*}
&\left |\int_{-\frac{\ell_0}{2}}^{\frac{\ell_0}{2}}F_{\e}((\zeta^2)_{\eta}(\tau,\cdot))D^{2n}\zeta(\tau,\cdot)d\eta\right |\\
\le & C\e^{\frac{3}{2}}|D^n\zeta(\tau,\cdot)|_2|D^{n+2}\zeta(\tau,\cdot)|_2^2
+C\e^{\frac{3}{2}}|D^{n+1}\zeta(\tau,\cdot)|_2^2|D^{n+2}\zeta(\tau,\cdot)|_2\\
&+C\e|D^n\zeta_{\eta}(\tau,\cdot)|_2|D^{n+1}\zeta(\tau,\cdot)|_2|D^{n+2}\zeta(\tau,\cdot)|_2
+C\e|D^{n+1}\zeta(\tau,\cdot)|_2^3\\
&+C\sqrt{\e}(1+\e)|D^n\zeta(\tau,\cdot)|_2^2|D^{n+2}\zeta(\tau,\cdot)|_2
+C|D^n\zeta(\tau,\cdot)|_2^2|D^{n+1}\zeta(\tau,\cdot)|_2\notag\\
\le &C\e^{\frac{3}{2}}|D^n\zeta(\tau,\cdot)|_2|D^{n+2}\zeta(\tau,\cdot)|_2^2
+C\e^{\frac{3}{2}}|D^{n+1}\zeta(\tau,\cdot)|_2^2|D^{n+2}\zeta(\tau,\cdot)|_2\\
&+C\e|D^{n+1}\zeta(\tau,\cdot)|_2^2|D^{n+2}\zeta(\tau,\cdot)|_2
+C\e|D^{n+1}\zeta(\tau,\cdot)|_2^2|D^{n+2}\zeta(\tau,\cdot)|_2\\
&+C\sqrt{\e}(1+\e)|D^n\zeta(\tau,\cdot)|_2^2|D^{n+2}\zeta(\tau,\cdot)|_2
+C|D^n\zeta(\tau,\cdot)|_2^2|D^{n+2}\zeta(\tau,\cdot)|_2\notag\\
\le & C\e^{\frac{3}{2}}|D^n\zeta(\tau,\cdot)|_2|D^{n+2}\zeta(\tau,\cdot)|_2^2
+C\left (\e^2|D^{n+1}\zeta(\tau,\cdot)|_2^4+\e|D^{n+2}\zeta(\tau,\cdot)|_2^2\right )\\
&+C|D^{n+2}\zeta(\tau,\cdot)|_2^2+C\e^2|D^{n+1}\zeta(\tau,\cdot)|_2^4
+C\e^2|D^{n+1}\zeta(\tau,\cdot)|_2^4+
C|D^{n+2}\zeta(\tau,\cdot)|_2^2\\
&+C|D^n\zeta(\tau,\cdot)|_2^4
+C\e|D^{n+2}\zeta(\tau,\cdot)|_2^2
+C|D^n\zeta(\tau,\cdot)|_2^4+C|D^{n+2}\zeta(\tau,\cdot)|_2^2,
\end{align*}
for any $\tau\in [0,T_1]$ and any $\e\in (0,1]$.
Now, Estimate \eqref{estim-bad-2} follows immediately.

To complete the proof, let us prove Estimate \eqref{estim-bad-3}.
From \eqref{estim-6} and the estimates
\begin{align*}
&|D^n(\Psi\zeta)(\tau,\cdot)|_2\le C|D^n\zeta(\tau,\cdot)|_2|D^n\Psi(\tau,\cdot)|_2,\\
&|D^{n+2}(\Psi\zeta)(\tau,\cdot)|_2\le C|D^{n+2}\zeta(\tau,\cdot)|_2|D^{n+2}\Psi(\tau,\cdot)|_2,
\end{align*}
(which can be proved using the same argument as in the proof of \eqref{estim-7} and \eqref{estim-8})
we get
\begin{align*}
&\left
|\int_{-\frac{\ell_0}{2}}^{\frac{\ell_0}{2}}F_{\e}((\Psi\zeta)_{\eta}(\tau,\cdot))D^{2n}\zeta(\tau,\cdot)
d\eta\right |\\
\le & \sqrt{2}\e^{\frac{3}{2}}|D^{n+2}\zeta(\tau,\cdot)|_2
|D^{n+2}(\Psi\zeta)(\tau,\cdot)|_2
+3\e|D^{n+1}\zeta(\tau,\cdot)|_2|D^{n+2}(\Psi\zeta)(\tau,\cdot)|_2\\
&+2\sqrt{\e}(1+\e)|D^{n+2}\zeta(\tau,\cdot)|_2
|D^n(\Psi\zeta)(\tau,\cdot)|_2\notag\\
&+\frac{3+\sqrt{2}}{4}|D^{n+1}\zeta(\tau,\cdot)|_2
|D^n(\Psi\zeta)(\tau,\cdot)|_2\\
\le &
C\e^{\frac{3}{2}}|D^{n+2}\Psi(\tau,\cdot)|_2|D^{n+2}\zeta(\tau,\cdot)|_2^2
+C\e |D^{n+2}\Psi(\tau,\cdot)|_2|D^{n+1}\zeta(\tau,\cdot)|_2|D^{n+2}\zeta(\tau,\cdot)|_2\\
&+C\sqrt{\e}|D^n\Psi(\tau,\cdot)|_2|D^n\zeta(\tau,\cdot)|_2
|D^{n+2}\zeta(\tau,\cdot)|_2\\
&+C|D^n\Psi(\tau,\cdot)|_2|D^n\zeta(\tau,\cdot)|_2|D^{n+1}\zeta(\tau,\cdot)|_2\\
\le & C\e^{\frac{3}{2}}|D^{n+2}\Psi(\tau,\cdot)|_2|D^{n+2}\zeta(\tau,\cdot)|_2^2
+C\e |D^{n+2}\Psi(\tau,\cdot)|_2|D^{n+1}\zeta(\tau,\cdot)|_2^2\\
&+C\e |D^{n+2}\Psi(\tau,\cdot)|_2|D^{n+2}\zeta(\tau,\cdot)|_2^2\\
&+C|D^n\Psi(\tau,\cdot)|_2|D^n\zeta(\tau,\cdot)|_2^2
+C\e |D^n\Psi(\tau,\cdot)|_2|D^{n+2}\zeta(\tau,\cdot)|_2^2\\
&+C\delta^{-1}|D^n\Psi(\tau,\cdot)|_2^2|D^n\zeta(\tau,\cdot)|_2^2
+C\delta|D^{n+2}\zeta(\tau,\cdot)|_2^2,
\end{align*}
for any $\tau\in [0,T_1]$, any $\e\in (0,1]$ and any $\delta>0$.
Estimate \eqref{estim-bad-3} follows taking $C\delta=1/4$.
This completes the proof.
\end{proof}

We are almost ready to write the crucial {\it a priori} estimate satisfied by $\zeta(\tau,\cdot)$.
For this purpose, we recall that
\begin{eqnarray*}
|D^{n+1}\psi|_2\le |D^n\psi|_2^{\frac{1}{2}}|D^{n+2}\psi|_2^{\frac{1}{2}},\qquad\;\,\psi\in H^{n+2}_{\sharp}.
\end{eqnarray*}
Applying this estimate to $D^{n+1}\zeta(\tau,\cdot)$
together with Young-inequality, yields
\begin{equation}
|D^{n+1}\zeta(\tau,\cdot)|_2^2\le |D^n\zeta(\tau,\cdot)|_2^2+\frac{1}{4}|D^{n+2}\zeta(\tau,
\cdot)|_2^2,
\label{interp}
\end{equation}
for any $\tau\in [0,T_1]$.
Combining Lemmata \ref{lemma-1},
\ref{lemma-2} and Estimate \eqref{interp} allows us to estimate
\begin{align}
&\frac{1}{2}\frac{d}{d\tau}\left (|D^n\zeta(\tau,\cdot)|_2^2
+4\e|D^{n+1}\zeta(\tau,\cdot)|_2^2
+(1+\e)|\sqrt{R_{\e}}D^n\zeta(\tau,\cdot)|_2^2\right )\notag\\
&+\left (\frac{15}{4}-C_n\e-\e K_2'(n,\Psi)
-C_n\e^{\frac{5}{2}}|D^n\zeta(\tau,\cdot)|_2
\right )|D^{n+2}\zeta(\tau,\cdot)|_2^2\notag\\
\le & K_1'(n,\Psi)
+\left (2+K_2'(n,\Psi)\right )
|D^n\zeta(\tau,\cdot)|_2^2
+\e K_2'(n,\Psi)|D^{n+1}\zeta(\tau,\cdot)|_2^2\notag\\
&+C_n\e |D^n\zeta(\tau,\cdot)|_2^4+C_n\e^3|D^{n+1}\zeta(\tau,\cdot)|_2^4,
\label{ineq-cruc}
\end{align}
for any $\tau\in [0,T_1]$. If we set
\begin{align*}
&A_{\e}(\tau)=|D^n\zeta(\tau,\cdot)|_2^2
+4\e|D^{n+1}\zeta(\tau,\cdot)|_2^2
+(1+\e)|\sqrt{R_{\e}}D^n\zeta(\tau,\cdot)|_2^2,\qquad\;\,\tau\in [0,T_1],\\[2mm]
&c_1=2K_1'(n,\Psi),
\\[1mm]
&c_2=4+2K_2'(n,\Psi),\\[1mm]
&c_3=2C_n,
\end{align*}
and assume $\e$ small enough such that
\begin{eqnarray*}
C_n\e+\e K_2'(n,\Psi)<\frac{3}{4},
\end{eqnarray*}
we can rewrite Inequality \eqref{ineq-cruc} in the more compact form
\begin{align*}
A'_{\e}(\tau)+\left (6-2C\e^{\frac{5}{2}}A_{\e}(\tau)
\right )|D^{n+2}\zeta(\tau,\cdot)|_2^2
\le c_1+c_2A_{\e}(\tau)+c_3\e(A_{\e}(\tau))^2,
\end{align*}
for any $\tau\in [0,T_1]$.

The following lemma allows us to estimate the function $A_{\e}$.

\begin{lemma}
\label{lem-5} Let $A_0$, $c_0$, $c_1$, $c_2$, $c_3$, $\e$, $T_0$, $T_1$ be positive constants with $T_1<T_0$.
Further, let $f_{\e}:[0,T_1]\to\mathbb R$ and $A_{\e}$
be positive functions of class $C([0,T_1])$ and $C^1([0,T_1])$,
respectively, that satisfy the inequalities
\begin{align*}
\left\{
\begin{array}{ll}
A_{\e}'(\tau)+(c_0-\e A_{\e}(\tau))f_{\e}(\tau)\le c_1+c_2A_{\e}(\tau)+c_3\e (A_{\e}(\tau))^2, & \tau\in [0,T_1],\\[1mm]
A_{\e}(0)=0.
\end{array}
\right.
\end{align*}
Then, there exist $\e_1=\e_1(T_0)\in (0,1)$ and
a constant $K=K(T_0)$ such that $A_{\e}(\tau)\le K$ for any
$\tau\in [0,T_1]$ and any $\e\in (0,\e_1]$.
\end{lemma}

\begin{proof}
The proof follows basically from \cite[Lemma 3.1]{BFHLS}, which
deals with the case when $f_{\e}\equiv 0$. Repeating the arguments
in that proof, we can easily show that $A_{\e}(\tau)\le
4c_1e^{c_2T_0}/(3c_2)$ for any $\tau\in [0,T_1]$ and any $\e\in
(0,\e_2(T_0)]$, where $\e_2(T_0)=3c_2^2/(16c_1c_3(e^{c_2T_0}-1))$.

Let us now consider the general case when $f_{\e}$ does not
identically vanish in $[0,T_1]$. We fix $\e_1(T_0)\le\e_2(T_0)$ such
that $3c_0c_2-4c_1e^{c_2T_0}\e_0>0$ and $\e\in (0,\e_0(T)]$. Since
$A_{\e}(0)=0$, there exists a maximal interval $[0,T_{\e})$ where
$c_0-\e A_{\e}>0$. We are going to prove that $T_{\e}=T_1$. For this
purpose, we observe that in $[0,T_{\e})$ the function $A_{\e}$
satisfies the inequality $A_{\e}'\le c_1+c_2A_{\e}+c_3\e A_{\e}^2$.
Hence, from the above result it follows that $A_{\e}(\tau)\le
(4c_1e^{c_2T_0})/(3c_2)$ for any $\tau\in [0,T_{\e}]$, so that
$c_0-\e A_{\e}(T_{\e})>0$. This clearly implies that $T_{\e}=T_1$.
\end{proof}

We are now in position to prove Theorem \ref{cor-apriori-estim}.
Applying Lemma \ref{lem-5} it follows immediately that
\begin{eqnarray*}
\sup_{\tau\in [0,T]}\left (|D^n\zeta(\tau,\cdot)|_2^2
+4\e|D^{n+1}\zeta(\tau,\cdot)|_2^2
+(1+\e)|\sqrt{R_{\e}}D^n\zeta(\tau,\cdot)|_2^2\right )\leq K_{1,n},
\end{eqnarray*}
for any $n=0,1,\ldots$, from which \eqref{estim-1} follows at once.

\subsection{Existence and uniqueness of a solution to Equation \eqref{pb-fin-ve} vanishing at $\tau=0$}
\label{Existence}
In this subsection we are devoted to prove the following theorem.

\begin{theorem}\label{zeta}
For any $T>0$, there exists $\e_0(T)>0$ such that
that, for any $0<\e \leq \e_0(T)$, Equation \eqref{eq-final-ve-1}
has a unique classical solution $\zeta$ on $[0,T]$, which vanishes at $\tau=0$.
\end{theorem}

{\bf Existence part.}
We prove the existence of a solution $\zeta$ to Equation \eqref{pb-fin-ve}, vanishing at $0$,
by a standard Faedo-Galerkin method.
Let us fix $\xi \in (I-\Pi)(H_{\sharp}^{s})$ and expand it into a Fourier series (see Section \ref{math-setting}) as follows:
\begin{align*}
\xi=\sum_{k=1}^{+\infty}{\widehat \xi}(k) w_k.
\end{align*}
For $N=1,2,\ldots$, we denote by $\Xi_N = P_N((I-\Pi)(H_{\sharp}^{s}))$  the projection of $(I-\Pi)(H_{\sharp}^{s})$ along the vector space spanned by the functions $w_1,\ldots,w_N$.

Let us look for a solution $\zeta_N\in\Xi_N$ to the variational problem
\begin{align}\label{variational-formulation}
\frac{\partial}{\partial\tau}
\int_{-\frac{\ell_0}{2}}^{\frac{\ell_0}{2}} B_{\e} (\zeta_N)\,\xi
d\eta =& \int_{-\frac{\ell_0}{2}}^{\frac{\ell_0}{2}} S(\zeta_N)\,\xi
d\eta
 +\int_{-\frac{\ell_0}{2}}^{\frac{\ell_0}{2}}
 \big\{M_{\e}((\Psi^2)_{\eta})-H_{\e}(\Psi_{\tau})\big\}\,\xi d\eta \notag\\
 &+\e  \int_{-\frac{\ell_0}{2}}^{\frac{\ell_0}{2}} F_{\e}((\zeta_N^2)_{\eta})\,\xi d\eta
 +2\int_{-\frac{\ell_0}{2}}^{\frac{\ell_0}{2}} F_{\e}((\Psi\zeta_N)_{\eta})\xi d\eta.
\end{align}
for all $\xi \in \Xi_N$. The problem is subject to the initial condition $\zeta_N(0,\cdot)\equiv 0$.
In terms of Fourier series, the variational formulation \eqref{variational} reads as follows:
\begin{align}\label{Fourier-variational-formulation}
\frac{d}{d\tau}\sum_{k=1}^{+\infty}b_{k,\e} {\widehat \zeta_N}(\cdot,k){\widehat \xi}(k) =&
\sum_{k=1}^{+\infty}s_k{\widehat \zeta_N}(\cdot,k){\widehat \xi}(k)\notag\\
&+ \sum_{k=1}^{+\infty}   \big\{m_{\e,k}\widehat{(\Psi^2)_{\eta}}
(\cdot,k) -h_{\e,k}\widehat{\Psi_{\tau}}(\cdot,k)\big\}{\widehat \xi}(k) \notag\\
&+\e \sum_{k=1}^{+\infty}f_{k,\e}\widehat{(\zeta_N^2)_{\eta}}(\cdot,k) {\widehat \xi}(k)
+2\sum_{k=1}^{+\infty}f_{\e,k}\widehat{(\Psi \zeta_N)_{\eta}}(\cdot,k) {\widehat \xi}(k).
\end{align}
Taking  $\xi=w_{j}$ ($j=1,\ldots,N$) in \eqref{Fourier-variational-formulation}, we see that the ${\widehat \zeta_N}(\cdot,k)$'s
verify a system of $N$ ordinary differential equations with zero initial data.
Hence, there exists a unique solution to System \eqref{Fourier-variational-formulation}, defined
on some maximal time interval $[0,T_N)$, where $T_N$ may also depend on $\e$.

Next, we take $\xi=\zeta_N$ in \eqref{variational-formulation}. The estimates of
Section \ref{Formal-estimates} remain valid also for the function $\zeta_N$. Writing such estimates, taking
as $T_1$ any number less than $T_N$ and then letting $T'\to T_N$, we thus get
\begin{equation}\label{estimate-sequence1}
\sup_{\tau\in [0,T_N)}
\int_{-\frac{\ell_0}{2}}^{\frac{\ell_0}{2}}|D^{n}\zeta_N(\tau,\cdot)|^2d\eta
+\e\int_{-\frac{\ell_0}{2}}^{\frac{\ell_0}{2}}|D^{n+1}\zeta_N(\tau,\cdot)|^2d\eta\le K_n,
\end{equation}
for any $n\geq 1$. From this estimate we infer that, whenever $0<\e \leq \e_1(T)$, the solution of the ODE system
can be extended up to $T$.

We now let $N \to +\infty$. For this purpose, we use Estimate \eqref{estimate-sequence1} with $n=4$.
It leads to the following facts:
\begin{enumerate}[\rm (i)]
\item
the sequence $(\zeta_N)_{N\in \mathbb{N}}$ is bounded in
$C([0,T];H_{\sharp}^5)$, with a bound possibly depending on $\e\in (0,\e_1(T)]$;
\item
the sequence $((\zeta_N)_{\tau})_{N\in \mathbb{N}}$ is bounded in  $C([0,T];H^3_{\sharp})$
with a bound possibly depending on $\e$.
\end{enumerate}

Property (i) follows immediately from \eqref{estimate-sequence1}. (Note that, if $n \geq 5$,
then the bound is uniform in $0<\e\leq \e_1(T)$.)
To prove property (ii), we observe that
\eqref{variational-formulation} may be rewritten as:
\begin{equation*}
\frac{\partial}{\partial\tau}B_{\e}(\zeta_N)=
P_N\bigg\{S(\zeta_N)+M_{\e}((\Psi^2)_{\eta})-H_{\e}(\Psi_{\tau})+\e F_{\e}((\zeta_N^2)_{\eta})
+2F_{\e}((\Psi\zeta_N)_{\eta})\bigg\},
\end{equation*}
and we use (i) and Proposition \ref{prop-H}.

Now, we can make the compactness argument work for any arbitrarily fixed $\e\in (0,\e_1(T)]$.
By the Sobolev embedding theorem, the sequences $(\zeta_N)_{N\in \mathbb{N}}$ and  $((\zeta_N)_{\tau})_{N\in\mathbb N}$
are bounded in $C([0,T];C^{9/2}_{\sharp})$ and in $C([0,T];C^{5/2}_{\sharp})$, respectively.
In particular, by interpolation we easily see that
$(D^l_{\eta}(\zeta_N)_{\tau})_{N\in\mathbb N}$ ($l=0,\ldots,4$) are bounded in $C^{1/9}([0,T]\times [-\ell_0/2,\ell_0/2])$.
Indeed, $C^4_{\sharp}$ belongs to the class $J_{8/9}$ between $C_{\sharp}$ and $C^{9/2}_{\sharp}$.
Hence, we can estimate
\begin{align*}
\|\zeta_N(\tau_2,\cdot)-\zeta_N(\tau_1,\cdot)\|_{C^4_{\sharp}}
&\le \|\zeta_N(\tau_2,\cdot)-\zeta_N(\tau_1,\cdot)\|_{\infty}^{\frac{1}{9}}
\|\zeta_N(\tau_2,\cdot)-\zeta_N(\tau_1,\cdot)\|_{C^{9/2}_{\sharp}}^{\frac{8}{9}}\\
&\le \|D_{\tau}\zeta_N\|_{\infty}^{\frac{1}{9}}\sup_{\tau\in [0,T]}\|\zeta_N(\tau,\cdot)\|_{C^{9/2}_{\sharp}}^{\frac{8}{9}}
|\tau_2-\tau_1|^{\frac{1}{9}},
\end{align*}
for any $\tau_1,\tau_2\in [0,T]$. Since the sequence $(\zeta_N)$ is bounded both in $C^{1/9}([0,T];C^4_{\sharp})$
and in $C([0,T];C^{9/2}_{\sharp})$, it is bounded in $C^{1/9}([0,T]\times [-\ell_0/2,\ell_0/2])$, as well.
Arzel\`a-Ascoli theorem, then implies that, up to a subsequence, $\zeta_N$ converges in $C([0,T];C^4_{\sharp})$
to a function $\zeta\in C([0,T];C^{9/2}_{\sharp})$. Similarly, $D^l_{\eta}((\zeta_{\tau})_N)_{N\in\mathbb N}$ ($l=0,1,2$) converges, up to a subsequence, to $D^l_{\eta}(\zeta_{\tau})$ $(l=0,1,2)$. Clearly, the function
$\zeta$ solves the equation \eqref{pb-fin-ve} and vanishes at $\tau=0$.

{\bf Uniqueness part.} Assume that $\zeta_1$ and $\zeta_2$ are two classical solutions to Equation
\eqref{eq-final-ve-1} which vanish at $\tau=0$. Then, the function $\chi:=\zeta_1-\zeta_2$ turns out to solve the equation
\begin{equation}\label{unique}
\frac{\partial}{\partial\tau}B_{\e}(\chi)  = S(\chi)
+\e F_{\e}((\chi(\zeta_1
+\zeta_2))_{\eta})+2F_{\e}((\Psi{\chi})_{\eta}).
\end{equation}
We multiply \eqref{unique} by $\chi$ and integrate over $[-\ell_0/2,\ell_0/2]$. We get
\begin{align*}
\int_{-\frac{\ell_0}{2}}^{\frac{\ell_0}{2}}
B_{\e}(\chi(\tau,\cdot))\chi(\tau,\cdot) d\eta
=&\int_{-\frac{\ell_0}{2}}^{\frac{\ell_0}{2}}
S(\chi(\tau,\cdot))\chi(\tau,\cdot)d\eta\notag\\
&+\e\int_{-\frac{\ell_0}{2}}^{\frac{\ell_0}{2}}F_{\e}((\chi(\tau,\cdot)(\zeta_1(\tau,\cdot)
+\zeta_2(\tau,\cdot)))_{\eta})
\chi(\tau,\cdot)d\eta\notag\\
&+2\int_{-\frac{\ell_0}{2}}^{\frac{\ell_0}{2}}
F_{\e}((\Psi{\chi}(\tau,\cdot))_{\eta})\chi(\tau,\cdot),
\end{align*}
for any $\tau\in [0,T]$. All the terms but
$\int_{-\ell_0/2}^{\ell_0/2}F_{\e}((\chi(\tau,\cdot)(\zeta_1(\tau,\cdot)
+\zeta_2(\tau,\cdot)))_{\eta})\chi(\tau,\cdot)d\eta$ have been
already estimated in Lemmata \ref{lemma-1} and \ref{lemma-2}. Hence,
we just need to estimate this latter term. For this purpose, we
observe that \eqref{estim-6} implies that
\begin{align}
&\left
|\int_{-\frac{\ell_0}{2}}^{\frac{\ell_0}{2}}F_{\e}((\chi(\zeta_1+\zeta_2))_{\eta}(\tau,\cdot))\chi(\tau,\cdot)d\eta\right
|
\label{uniqueness-variat-1}\\
\le &2\sqrt{2}\e^{\frac{3}{2}}|D^2(\chi(\zeta_1+\zeta_2))(\tau,\cdot)|_2|D^2\chi(\tau,\cdot)|_2
+3\e |D^2(\chi(\zeta_1+\zeta_2))(\tau,\cdot)|_2|D\chi(\tau,\cdot)|_2\notag\\
&+4\sqrt{\e}|\chi(\tau,\cdot)(\zeta_1+\zeta_2)(\tau,\cdot)|_2|D^2\chi(\tau,\cdot)|_2\notag\\
&+\frac{3+\sqrt{2}}{4}|\chi(\tau,\cdot)(\zeta_1+\zeta_2)(\tau,\cdot)|_2|D\chi(\tau,\cdot)|_2,\notag
\end{align}
for any $\tau\in [0,T]$.
By the a priori estimates \eqref{main-estimate} with $n=1$, we infer that
\begin{eqnarray*}
|D(\zeta_1+\zeta_2)(\tau,\cdot)|_2^2+\e|D^2(\zeta_1+\zeta_2)(\tau,\cdot)|_2^2\le 2K_1,\qquad\;\,\tau\in [0,T].
\end{eqnarray*}
Therefore,
\begin{align*}
&|\chi(\tau,\cdot)(\zeta_1+\zeta_2)(\tau,\cdot)|_2\le C_1|\chi(\tau,\cdot)|_2|D(\zeta_1+\zeta_2)(\tau,\cdot)|_2
\le 2C_1K|\chi(\tau,\cdot)|_2,\\[2mm]
&|D^2(\chi(\zeta_1+\zeta_2))(\tau,\cdot)|_2\le C_1|D^2(\zeta_1+\zeta_2)(\tau,\cdot)|_2|D\chi(\tau,\cdot)|_2\\
&\qquad\qquad\qquad\qquad\qquad\;+C_1|D(\zeta_1+\zeta_2)(\tau,\cdot)|_2|D^2\chi(\tau,\cdot)|_2\\
&\qquad\qquad\qquad\qquad\quad\;\,\le 2C_1K\e^{-1}|D\chi(\tau,\cdot)|_2
+2C_1K|D^2\zeta(\tau,\cdot)|_2,
\end{align*}
for any $\tau\in [0,T]$ and some positive constant $C_1$, depending on $\ell_0$ only.
We can thus continue \eqref{uniqueness-variat-1} getting
\begin{align*}
&\left |\int_{-\frac{\ell_0}{2}}^{\frac{\ell_0}{2}}F_{\e}((\chi(\zeta_1+\zeta_2))_{\eta}(\tau,\cdot))\chi(\tau,\cdot)d\eta\right |\notag\\
\le & C_K\sqrt{\e}|D\chi(\tau,\cdot)|_2|D^2\chi(\tau,\cdot)|_2
+C_K\e^{\frac{3}{2}}|D^2\chi(\tau,\cdot)|_2^2
+C_K\sqrt{\e}|\chi(\tau,\cdot)|_2|D\chi(\tau,\cdot)|_2\\
&+C_K|D\chi(\tau,\cdot)|_2^2
+C_K|\chi(\tau,\cdot)|_2|D\chi(\tau,\cdot)|_2\\
\le & C_K\e |D^2\chi(\tau,\cdot)|_2^2+
C_K|D\chi(\tau,\cdot)|_2^2+
C_K |D^2\chi(\tau,\cdot)|_2^2,
\end{align*}
for any $\tau\in [0,T]$ and some positive constant $C_K$, depending on $K$ only.
Hence, combining this estimate with \eqref{estim-1} and \eqref{estim-bad-3} yields
\begin{align*}
&\frac{1}{2}\frac{d}{d\tau}|\chi(\tau,\cdot)|_2^2 +2\e
\frac{d}{d\tau}|D\chi(\tau,\cdot)|_2^2
+\frac{1+\e}{2}\frac{d}{d\tau}|\sqrt{R_{\e}}\chi(\tau,\cdot)|_2^2
+M_{\e,K}|D^2\chi(\tau,\cdot)|_2^2\\
\le &\left (K_2(0,\Psi)+C_K\e+1\right )|\chi(\tau,\cdot)|_2^2
+\e \left (C_K+K_2(0,\Psi)\right )|D\chi(\tau,\cdot)|_2^2,
\end{align*}
for any $\tau\in [0,T]$,
where $M_{\e,K}=\frac{15}{4}-C_K\e^2-K_2(0,\psi)\e^{\frac{3}{2}}$.
Up to replacing $\e_1(T)$ with a smaller value $\e_0(T)$, if needed, we can assume that
$M_{\e,K}\le 0$ for any $\e\in (0,\e_0(T)]$.
Now, Gronwall lemma applies and yields $\zeta\equiv 0$ since $\zeta(0,\cdot)=0$.

\subsection{Proof of Theorem $\ref{thm:2}$}

We now return to $\rho$ and to Problem  \eqref{pb-fin-ve}.
This can be done as in the proof of Theorem \ref{thm:1}. The idea is simple: we look for
$\rho$ as $\rho(\tau,\eta) = \chi(\tau,\eta) + p(\tau)w_0$, where $\chi$ has zero average. More precisely,
we set $\chi={\mathscr P}(\zeta)$, where operator ${\mathscr P}$ is defined by \eqref{operat-P}.
A simple computation shows that
\begin{eqnarray*}
B_{\e}(\chi)+H_{\e}(\Phi_{\eta})-S(\chi)-M_{\e}((\Phi_{\eta})^2)
-\e F_{\e}((\chi_{\eta})^2)-2F_{\e}(\Phi_{\eta}\chi_{\eta})
\end{eqnarray*}
is independent of $\eta$. Since $\chi\in (I-\Pi)(L^2)$, this means that
\begin{align*}
B_{\e}(\chi)+H_{\e}(\Phi_{\eta})=&S(\chi)+(I-\Pi)(M_{\e}((\Phi_{\eta})^2))
+\e (I-\Pi)(F_{\e}((\chi_{\eta})^2))\\
&+2(I-\Pi)(F_{\e}(\Phi_{\eta}\chi_{\eta})).
\end{align*}
Let us now denote by $p:[0,T]\to\mathbb R$ the solution to the Cauchy problem
\begin{equation}
\left\{
\begin{array}{ll}
\displaystyle\frac{dp}{d\tau}=-\Pi(H_{\e}(\Phi_{\eta}))+\Pi(M_{\e}((\Phi_{\eta})^2))
+\e\Pi(F_{\e}((\chi_{\eta})^2))+2\Pi(F_{\e}(\Phi_{\eta}\chi_{\eta})),\\[3mm]
p(0)=0.
\end{array}
\right.
\label{pb-Pi-part}
\end{equation}
If we now set $\rho=p+\chi$, we immediately see that
$\rho(0,\cdot)=0$ and $\rho$ solves equation \eqref{pb-fin-ve}.

Clearly, this function is the unique solution to the Equation
\eqref{pb-fin-ve} which vanishes at $\tau=0$. Indeed, if $\rho_1$
and $\rho_2$ are two of such solutions, then the functions
$\zeta_1:=D_{\eta}\rho_1$ and $\zeta_2:=D_{\eta}\rho_2$ solve
Equation \eqref{eq-final-ve-1} and vanish at $\tau=0$. By the above
results, $\zeta_1$ and $\zeta_2$ do agree. This means that
$(I-\Pi)(\rho_1)\equiv (I-\Pi)(\rho_2)$. But then also $\Pi(\rho_1)$
and $\Pi(\rho_2)$ agree, since, as Problem \eqref{pb-Pi-part} shows,
$\Pi(\rho_1)$ and $\Pi(\rho_2)$ are uniquely determined by
$(I-\Pi)(\rho_1)$.

To complete the proof of Theorem \ref{thm:2}, let
us check that  there exists $M>0$ such that
\begin{equation}
\sup_{{\tau\in [0,T]}\atop{\eta\in
[-\ell_0/2,\ell_0/2]}} |\rho(\tau,\eta)| \leq M,
\label{final-apriori}
\end{equation}
uniformly in $0<\e\leq \e_0(T)$.
Applying the a priori estimates in Theorem \ref{cor-apriori-estim}
(here $n=0$ is enough)
and using \eqref{operat-P}, one can easily show that
\begin{eqnarray*}
\|(I-\Pi)(\rho)\|_{\infty}=\|{\mathscr P}(\zeta)\|_{\infty} \le
(1+\ell_0)\sqrt{\ell_0K_0}.
\end{eqnarray*}

As far as the component of $\rho$ along $\Pi(L^2)$ is concerned (which we still denote by $p$),
we observe that (see \eqref{fk}, \eqref{hk} and \eqref{mk})
\begin{align*}
&\Pi(H_{\e}(\Phi_{\eta}))=\Pi(M_{\e}((\Phi_{\eta})^2))=0,\\
&\Pi(F_{\e}((\chi_{\eta})^2))=-\frac{1}{2}\Pi((\chi_{\eta})^2),\\
&\Pi(F_{\e}(\Phi_{\eta}\chi_{\eta}))=-\frac{1}{2}\Pi(\Phi_{\eta}\chi_{\eta}),
\end{align*}
and we can estimate
\begin{align*}
|\Pi(F_{\e}((\chi_{\eta}(\tau,\cdot))^2))|\le
&\frac{1}{2}|\chi_{\eta}(\tau)|_2
\le\frac{1}{2}\sup_{\tau\in [0,T]}|\zeta(\tau,\cdot)|_2^2\le \frac{1}{2}K_0,\\[2mm]
|\Pi(F_{\e}(\Phi_{\eta}(\tau,\cdot)\chi_{\eta}(\tau,\cdot)))|\le
&\frac{1}{2}\int_{-\frac{\ell_0}{2}}^{\frac{\ell_0}{2}}
|\Phi_{\eta}(\tau,\cdot)\chi_{\eta}(\tau,\cdot)|d\eta\\
\le &\frac{1}{2}|\Phi_{\eta}(\tau,\cdot)|_2|\zeta(\tau,\cdot)|_2\\
\le &\frac{1}{2}\sqrt{K_0}\sup_{\tau\in [0,T]}|\Phi_{\eta}(\tau,\cdot)|_2,
\end{align*}
for any $\tau\in [0,T]$. It thus follows from \eqref{pb-Pi-part} that
\begin{align*}
|p(\tau)|&\le \e\int_0^{\tau}|\Pi(F_{\e}((\chi_{\eta}(\tau,\cdot))^2))
+2\Pi(F_{\e}(\Phi_{\eta}(\tau,\cdot)\chi_{\eta}(\tau,\cdot)))|d\tau\\
&\le \frac{1}{2}K_0T+\sqrt{K_0}T\sup_{\tau\in [0,T]}|\Phi_{\eta}(\tau,\cdot)|_2,
\end{align*}
for any $\tau\in [0,T]$. Estimate \eqref{final-apriori} now follows immediately.

Finally, coming back to Problem \eqref{psd-abstract} and setting
$\ell_{\e}=\ell_0/\sqrt{\e}$ and $T_{\e}=T/\e^{2}$, one can easily conclude that, for any $\e\in (0,\e_0]$, such a problem admits a
unique classical solution $\varphi$. Moreover,
\begin{eqnarray*}
\|\varphi(t,\cdot)-\e\Phi(t\e^2,\sqrt{\e}\,\cdot)\|_{C([-\ell_{\e}/2,\ell_{\e}/2])}\le
\varepsilon^2 M,\qquad\;\,t\in [0,T_{\e}].
\end{eqnarray*}
This accomplishes the proof of Theorem \ref{thm:2}.

\appendix

\section{Proof of Theorem \ref{thm-4.1}}
\label{sect-appendix}
\setcounter{equation}{0}

Showing that Problem \eqref{4order} admits a unique solution
$\Phi\in C^1([0,T_0];L^2)\cap C([0,T_0];H^4_{\sharp})$ for some
$T_0>0$ is an easy task. Indeed, the operator $A:H^2_{\sharp}\to
L^2$ is sectorial in $L^2$ as it has been already remarked. By
\cite[Prop. 2.4.1 \& 2.4.4]{lunardi} the operator $S=-4A^2-A$ is
sectorial in $L^2$ with domain $H^4_{\sharp}$. Classical results for
semilinear equations associated with sectorial operators show that
the Cauchy problem \eqref{4order} admits a unique solution $\Phi$
with the above regularity properties. (See e.g., \cite[Prop.
7.1.10]{lunardi}.) $\Phi$ turns out to be a fixed point of the
operator $\Gamma$, formally defined by
\begin{eqnarray*}
(\Gamma(\Phi))(\tau,\cdot)=e^{\tau S}\Phi_0
+\int_0^{\tau}e^{(\tau-s)S}(\Phi_{\eta}(s,\cdot))^2ds,\qquad\;\,\tau>0,
\end{eqnarray*}
where $\{e^{\tau S}\}$ denotes the semigroup generated by $S$.

Using a classical continuation argument, we can extend $\Phi$ to a maximal
domain $[0,T_{\max})$ with a function (still denoted by $\Phi$) which
belongs to $C^1([0,T_{\max});L^2)\cap C([0,T_{\max});H^4_{\sharp})$.

Let us regularize $\Phi$. Suppose that $\Phi_0\in H^5_{\sharp}$.
Note that $S$ commutes with $D_{\eta}$. Hence,
\begin{eqnarray*}
\Phi_{\eta}(\tau,\cdot)=e^{\tau S}(D_{\eta}\Phi_0)
+\int_0^{\tau}e^{(\tau-s)S}D_{\eta}(\Phi_{\eta}(s,\cdot))^2ds,\qquad\;\,\tau\in
[0,T_{\max}).
\end{eqnarray*}
Since $\Phi\in C^1([0,T_{\max});L^2)\cap C([0,T_{\max});H^4_{\sharp})$ and
$H^k_{\sharp}$ belongs to the class $J_{k/4}$ between $L^2$ and $H^4$, we can estimate
\begin{align*}
\|\Phi(\tau_2,\cdot)-\Phi(\tau_1,\cdot)\|_k\le &
|\Phi(\tau_2,\cdot)-\Phi(\tau_1,\cdot)|_2^{1-\frac{k}{4}}
\|\Phi(\tau_2,\cdot)-\Phi(\tau_1,\cdot)\|_4^{\frac{k}{4}}\\
\le & 2\|\Phi_{\tau}\|_{C([0,T_1];L^2)}
\|\Phi\|_{C([0,T_1];H^4_{\sharp})}^{\frac{k}{4}}|\tau_2-\tau_1|^{1-\frac{k}{4}},
\end{align*}
for any $\tau_1,\tau_2\in [0,T_1]$ and any $T_1<T_{\max}$.
Therefore, by the Sobolev embedding theorem, we can estimate
\begin{align*}
|D_{\eta}(\Phi_{\eta}(\tau_2,\cdot))^2-D_{\eta}(\Phi_{\eta}(\tau_1,\cdot))^2|_2
\le &|\Phi_{\eta}(\tau_2,\cdot)|_{\infty}|\Phi_{\eta\eta}(\tau_2,\cdot)-
\Phi_{\eta\eta}(\tau_1,\cdot)|_2\\
&+|\Phi_{\eta\eta}(\tau_2,\cdot)|_{\infty}|\Phi_{\eta}(\tau_2,\cdot)-
\Phi_{\eta}(\tau_1,\cdot)|_2\\
\le &
C_{T_1}|\tau_2-\tau_1|^{\frac{1}{2}},
\end{align*}
for any $\tau_1$ and $\tau_2$ as above.
This shows that $D_{\eta}(\Phi_{\eta}(\tau_2,\cdot))^2$ belongs to
$C^{1/2}([0,T_1];L^2)$ for any $T_1<T_{\max}$.
Theorem 4.3.1 of \cite{lunardi} implies that $\Phi_{\eta}\in
C^1([0,T_{\max});L^2)\cap C([0,T_{\max});H^4_{\sharp})$.
In particular, $\Phi_{\tau}$ belongs to $C([0,T_{\max});H^1_{\sharp})$.
It follows that $\Phi_{\tau\eta}\equiv\Phi_{\eta\tau}$.
Iterating this argument shows that, if $\Phi_0\in H^m_{\sharp}$ for
some $m\in\mathbb N$ such that $m>4$, then $\Phi\in C([0,T_{\max});H^m_{\sharp})$ and
$\Phi_{\tau}\in C([0,T_{\max});H^{m-4}_{\sharp})$.

The rest of the proof is devoted to show that $T_{\max}=+\infty$.
We adapt the arguments in \cite[Thm. 2.4]{tadmor}.
The main
step is the {\it a priori estimate}
\begin{equation}
|\Phi_{\eta}(\tau,\cdot)|_2\le
e^{\frac{13}{6}\tau}|D_{\eta}\Phi_0|_2,\qquad\;\,\tau\in [0,T_{\max}).
\label{apriori-z}
\end{equation}
For this purpose, we introduce the function $v$, defined by
$v(\tau,\eta)=e^{-2\tau}\Phi_{\eta}(\tau,\eta)$ for any
$(\tau,\eta)\in [0,T_{\max})\times [-\ell_0/2,\ell_0/2]$. The smoothness of
$\Phi$ implies that $v\in C^{1,4}([0,T_{\max})\times
[-\ell_0/2,\ell_0/2])$, solves the parabolic equation
\begin{equation}
v_\tau= -3v_{\eta\eta\eta\eta}-v_{\eta\eta} - e^{2\tau}vv_\eta-2v,
\label{parab-eq}
\end{equation}
and satisfies the boundary conditions
$D_{\eta}^{(k)}v(\tau,-\ell_0/2)=D_{\eta}^{(k)}v(\tau,\ell_0/2)$ for any
$\tau\in [0,T)$ and $k=0,1,2,3$. Multiplying both the sides of
\eqref{parab-eq} by $v(\tau,\cdot)$, integrating on
$(-\ell_0/2,\ell_0/2)$ and observing that the integral over
$(-\ell_0/2,\ell_0/2)$ of $(v(\tau,\cdot))^2v_{\eta}(\tau,\cdot)$
vanishes for any $\tau\in [0,T_{\max})$, we get
\begin{equation}
\frac{d}{d\tau}|v(\tau,\cdot)|_2^2
+3|v_{\eta\eta}(\tau,\cdot)|_2^2-|v_{\eta}(\tau,\cdot)|_2^2+2|v(\tau,\cdot)|_2^2=0,
\qquad\;\,\tau\in [0,T_{\max}).
\label{energy-1}
\end{equation}
In view of the estimate
\begin{eqnarray*}
|v_{\eta}(\tau,\cdot)|_2^2\le
|v(\tau,\cdot)|_2|v_{\eta\eta}(\tau,\cdot)|_2 \le
3|v_{\eta\eta}(\tau,\cdot)|_2^2+\frac{5}{3}|v(\tau,\cdot)|_2^2,\qquad\;\,\tau\in
[0,T_{\max}),
\end{eqnarray*}
Formula \eqref{energy-1} leads us to the inequality
\begin{eqnarray*}
\frac{d}{d\tau}|v(\tau,\cdot)|_2^2+\frac{1}{3}|v(\tau,\cdot)|_2^2\le
0, \qquad\;\,\tau\in [0,T_{\max}),
\end{eqnarray*}
from which Estimate \eqref{apriori-z} follows at once.

We can now complete the proof.
For this purpose, let us consider the function $\Psi$, defined by
$\Psi(\tau,\eta)=\Phi(\tau,\eta)-\Pi(\Phi(\tau,\cdot))$ for any
$\tau\in [0,T_{\max})$ and any $\eta\in [-\ell_0/2,\ell_0/2]$. Applying Poincar\'e
inequality, we get
\begin{equation}
|\Phi(\tau,\cdot)-\Pi(\Phi(\tau,\cdot))|_2\le\sqrt{\ell_0}e^{\frac{13}{6}\tau}|D_{\eta}\Phi_0|_2,
\qquad\;\,\tau\in [0,T_{\max}). \label{I-P-sup}
\end{equation}

Let us now show that the function $\tau\mapsto\Pi(\Phi(\tau,\cdot))$
satisfies a similar estimate. For this purpose, we fix $\tau\in [0,T_{\max})$ and apply the operator $\Pi$ to both the
sides of \eqref{4order}. Since $\Phi$ and its derivatives satisfy
periodic boundary conditions,
\begin{eqnarray*}
\frac{d}{d\tau}\Pi(\Phi(\tau,\cdot))= \Pi(\Phi_{\tau}(\tau,\cdot))=
-\frac{1}{2\ell_0}\Pi((\Phi_{\eta}(\tau,\cdot))^2),
\end{eqnarray*}
for any $\tau\in [0,T_{\max})$. Taking \eqref{apriori-z} into account, we
can then estimate
\begin{eqnarray*}
\left |\frac{d}{d\tau}\Pi(\Phi(\tau,\cdot))\right |
\le\frac{1}{2\ell_0}e^{\frac{13}{3}\tau}|D_{\eta}\Phi_0|_2^2,\qquad\;\,\tau\in
[0,T_{\max}).
\end{eqnarray*}
Hence,
\begin{equation}
|\Pi(\Phi(\tau))|\le|\Pi(\Phi_0)|+\int_0^{\tau}\left
|\frac{d}{d\tau}\Pi(\Phi(\tau,\cdot))\right |d\tau\le |\Pi(\Phi_0)|
+\frac{3}{26\ell_0}|D_{\eta}\Phi_0|_2^2e^{\frac{13}{3}\tau},
\label{Pi-sup}
\end{equation}
for any $\tau\in [0,T_{\max})$. Estimates \eqref{I-P-sup} and
\eqref{Pi-sup} show that $\Phi$ is bounded in $[0,T_{\max})$ with values in $L^2$. Therefore, we can apply \cite[Prop.
7.2.2]{lunardi} with $\gamma=1/2$, $\alpha=1/4$, $X_{1/4}=H^1_{\sharp}$, which implies
that $T_{\max}=+\infty$.

\section*{Acknowledgments}
C.-M. B thanks the VU University Amsterdam and the Department of
Mathematics of Parma for their kind hospitality during his visits.
L. L. was a visiting professor at the University of Bordeaux 1 in
2008-2009. He greatly acknowledges the Institute of Mathematics of
Bordeaux for the kind hospitality during his visits. The work of G.
I. S. was supported in part by the US-Israel Binational Science
Foundation (Grant 2006-151), and the Israel Science Foundation
(Grant 32/09).

\medskip
Received xxxx 20xx; revised xxxx 20xx.
\medskip

\end{document}